\documentclass[11pt]{article}
\usepackage{caption}
\usepackage{subfigure}
\usepackage{graphicx}
\usepackage{epsfig}
\usepackage{enumerate}
\usepackage{amsmath}
\usepackage{amsfonts}
\usepackage{amssymb}
\usepackage{dsfont}
\usepackage{mathrsfs}
\usepackage{bbm}
\usepackage{hyperref}
\usepackage{multirow}
\usepackage{algorithmic}
\usepackage{mathrsfs}
\usepackage[dvipsnames,usenames]{color}
\usepackage{cleveref}
\Crefname{ALC@unique}{Line}{Lines}

\renewcommand{\Box}{\framebox{\rule{0.3em}{0.0em}}}

\newtheorem{theorem}{Theorem}[section]
\newtheorem{lem}{Lemma}[section]
\newtheorem{prop}{Proposition}[section]
\newtheorem{ex}{Example}[section]

\newtheorem{definition}{Definition}[section]
\newtheorem{algorithm}{Algorithm}[section]
\newtheorem{assumption}{Assumption}[section]

\newcommand{\setd}{{ d \kern -.15em l}}
\newcommand{\hatsetd}{ d \hat{\kern -.15em l }}
\newcommand{\dd}{\mathsf {d\kern -0.07em l}} 

\newcommand{\bgeqn}{\begin{eqnarray}}
\newcommand{\edeqn}{\end{eqnarray}}
\newcommand{\bgeq}{\begin{eqnarray*}}
\newcommand{\edeq}{\end{eqnarray*}}
\newcommand{\bec}{\begin{center}}
\newcommand{\enc}{\end{center}}
\newcommand{\R}{{\rm I\!R}}

\newcommand{\F}{{\cal F}}

\newcommand{\be}{\begin{equation}}
\newcommand{\ee}{\end{equation}}

\def\bbe{{\Bbb{E}}} 

\renewcommand{\Box}{\hfill \rule{2.3mm}{2.3mm}}

\title{
Preference Robust Modified Optimized Certainty Equivalent 
}

\author{
Qiong Wu\footnote{Department of Systems Engineering and Engineering Management, The Chinese University of Hong Kong, Hong Kong. Email: qiwu@se.cuhk.edu.hk
}
\,\,and\,\,
Huifu Xu\footnote{Department of Systems Engineering and Engineering Management, The Chinese University of Hong Kong, Hong Kong. Email: hfxu@se.cuhk.edu.hk}
}

\date{\today}

\begin{document}

\maketitle

\begin{abstract}

Ben-Tal and Teboulle \cite{BTT86} introduce the concept of optimized certainty equivalent (OCE) of an uncertain outcome
as the maximum present value of a combination of the
cash to be taken out from
the uncertain
income at present
and the expected utility value of the remaining
uncertain income.
In this paper, we consider two
variations of the OCE. First, we
introduce a modified OCE by
maximizing the combination of
the utility of the cash and the expected utility of the remaining
uncertain income so that the combined quantity is in a unified
utility value. Second, we consider
a situation where the true utility function is unknown
but it is possible to use partially available information to construct a
set of plausible utility functions.
To mitigate the risk arising from the ambiguity, we introduce a robust model where
the modified OCE is based on the worst-case utility function from the ambiguity set.
In the case when the ambiguity set of utility functions is constructed by
a Kantorovich ball centered at a nominal utility function,
we show how  the modified OCE and the corresponding worst case utility function
can be identified by solving two linear programs alternatively.
We also show the robust modified OCE is statistically robust
in a data-driven environment where the underlying data are potentially contaminated.
Some preliminary numerical results  are reported to demonstrate the performance of the modified OCE and the robust modified OCE model.

\vspace{0.3cm}

\end{abstract}

\section{Introduction}
\label{sec:intro}
Let $(\Omega,\mathcal F, \mathbb{P})$ be a probability space with $\sigma$ algebra $\mathcal F$ and probability measure $\mathbb{P}$
and $\xi:(\Omega,\mathcal F, \mathbb{P})\to \R$ be a random variable representing future income of a decision maker (DM).
The {\em optimized
	certainty equivalent} of $\xi$
is defined as
\begin{equation}
{\rm (OCE)}\quad \quad \displaystyle{S_u(\xi):=\sup_{x\in \R} \;\; \{x+\bbe_P[u(\xi-x)]\}},
\label{eq:OCE}
\end{equation}
where $u:\R\to \R$ is the decision maker's utility function and
$P:=\mathbb{P}\circ\xi^{-1}$ is the probability measure on $\R$ induced by $\xi$.
The concept is first introduced by Ben-Tal and Teboulle \cite{BTT86} and closely related to other notions of certainty equivalent and risk measures, see \cite{BTT07} for a comprehensive discussion.
The economic interpretation of this notion is that the decision maker may need to consume part of $\xi$ at present, denoted by $x$, the sure present value of $\xi$ under the consumption plan becomes $x+\bbe_{P}[u(\xi-x)]$, and
the optimized certainty equivalent
$S_u(\xi)$ gives rise to the optimal allocation of the consumption which maximizes the sure present value of $\xi$.
As a  measure, it enjoys
a number of nice properties including
constancy ($S_u(C)=C$ for constant $C$),
risk aversion ($S_u(\xi)\leq \bbe_P[\xi]$) and
translation invariance ($S_u(\xi+C)=S_u(\xi)+C$). In particular, if $u$ is a normalized exponential utility function, it coincides with
the classical certainty equivalent $u^{-1}(\bbe_P[u(\xi)])$ in the literature of economics.
{\color{black}Moreover,  if 
$u(t)=-\frac{1}{\alpha}(-t)_+$ where $\alpha\in (0,1)$ and  $(t)_+=\max\{t,0\}$ for $t\in\R$, then the OCE effectively recovers the conditional value-at-risk (CVaR):  
\begin{eqnarray}
\displaystyle
S_u(\xi)
&=& \sup_{x\in\R} \;\; \left\{x-\frac{1}{\alpha}\bbe_P[(-\xi+x)_+]\right\}\nonumber\\
&=&-\inf_{x\in\R} \;\; \left\{x+\frac{1}{\alpha}\bbe_P[(-\xi-x)_+]\right\}=
-\text{ CVaR}_\alpha(-\xi).
\label{eq:CVaR}
\end{eqnarray}
The last equality is Rockafellar and Uryasev's formulation of CVaR, see \cite{BTT07,RoU00}.
Since CVaR is average of quantile, 
it is also known as average value-at-risk (AVaR), tail value-at-risk (TVaR) 
and expected shortfall, see \cite{PfR07,OgR02}. 

}


In this paper, we revisit the subject OCE from two perspectives.
One is to consider a modified version of
the optimized certainty equivalent
\begin{equation}
{\rm (MOCE)}\quad \quad   \displaystyle{M_u(\xi):=\sup_{x\in\R} \;\; \{u(x)+\bbe_P[u(\xi-x)]\}}.
\label{eq:OCE-new}
\end{equation}
The modification is motivated to align the sure present value
of $\xi$ to the expected utility theory \cite{VNM47}
by considering
the {\em utility} of present consumption $u(x)$ instead of the monetary value $x$. 
{\color{black}
Recall that in Von Neumann-Morgenstern 
expected utility theory \cite{VNM47},  the utility function is used to represent the decision maker's 
preference relation over a prospect 
space including both random and deterministic prospects, and such representation is unique
up to positive linear transformation. 
This means we can use both $u$ and $100u$ to represent the DM's preference. However, 
the two utility functions would lead to completely 
different optimal values and optimal solutions 
in the OCE model. In contrast, 
 the optimal solution is not affected in the MOCE model, 
 and the optimal value is only affected by the same scale of the utility function. 
In our view, this kind of ``invariance'' of the optimal allocation $x^*$ and ``scalability'' w.r.t. the utility function is important 
because the optimal decision on the allocation/consumption $x$
should be
determined by the DM's 
risk preference 
irrespective of its equivalent representations.


The modified OCE model
may be regarded as a special case of the well known consumption/investment models 
in economics \cite{VNM47,Hak75,Coc09} where $u(x)$ is the utility of the current consumption/investment
whereas $\bbe[u(\xi-x)]$ is the expected utility of the remaining asset to be consumed/invested in future. 
In these models, the utility functions
for the current consumption and future consumption are identical.
It is also possible to use different utility functions when the consumption 
at present is used for a new investment or production. 
}

{\color{black}
The other is to consider a situation where
the decision maker's utility function $u(\cdot)$ is ambiguous, in other words, there is incomplete information to identify a utility function $u$ which captures the decision maker's true utility preference.}
Consequently we propose to consider a robust optimized certainty equivalent measure
\begin{equation}
{\rm (RMOCE)}\quad \quad  \displaystyle{R(\xi):=\sup_{x\in\R} \inf_{u\in {\cal U}}\;\; \{u(x)+\bbe_P[u(\xi-x)]\}},
\label{eq:OCE-new-robust}
\end{equation}
where ${\cal U}$ is a set of plausible utility functions
consistent with the observed utility preferences of the decision maker. The definition is in line with the philosophy of robust optimization where the optimized certainty equivalent value is based on the worst case utility function from set  ${\cal U}$
to mitigate the risk arising from potential inaccurate use or misuse of the utility function.
By convention, we call ${\cal U}$ the ambiguity set.
In the case that ${\cal U}$ is a singleton, RMOCE reduces to MOCE.
Note that
RMOCE should be differentiated from
the  distributionally robust formulation of
OCE by Wisemann et al.~\cite{WKS14} where
the focus is on the ambiguity of $P$.
In decision analysis, $P$ is known as a
decision maker's belief of the state of nature
whereas $u$
characterizes the decision maker's
taste for risk/utility.
The RMOCE model concerns the ambiguity of
decision maker's taste rather than belief.

Ambiguity of utility preference is a well discussed topic in behavioural economics.
For instances, Thurstone \cite{Thu27} regards such ambiguity as a lack of accurate description
of human behaviour.
Karmarkar \cite{Kar78} and  Weber \cite{Web87}
ascribe the ambiguity to cognitive difficulty and incomplete information. The ambiguity may also arise in
the decision making problems which
involve several stakeholders who fail to reach a consensus.
Parametric and non-parametric approaches have subsequently been proposed to assess the true utility function, including
discrete choice models (Train \cite{Tra09}), standard and paired gambling approaches for preference comparisons and certainty
equivalence (Farquhar \cite{Far84}), we refer
readers to Hu et al.~\cite{HuBM18} for an excellent overview on this.

{\color{black}
In decision making under uncertainty, 
a decision maker may choose the worst case utility function among a set of 
plausible utility functions representing his/her risk preference 
to mitigate the overall risk. This kind of research may be traced back to 
Maccheroni \cite{Mac02}. Cerreia-Vioglio et al.~\cite{CDO15}
seem to be the first to 
investigate ambiguity of decision maker's utility function in the 
certainty equivalent model $u^{-1}(\bbe[u(\xi)])$ by
considering the worst-case certainty equivalent 
from a given set of utility functions in their cautious expected utility model.
They show that the DM's risk preference can be represented by a worst-case certainty equivalent if and only if they are given by a binary relation satisfying the weak order, continuity, weak monotonicity and negative certainty independence (NCI) (NCI states that if
a sure outcome is not 
enough to compensate the DM for a risky prospect, then its mixture with another lottery which reduces the certainty appeal, will not be more attractive than the same mixture of the risky prospect and the lottery).

Armbruster and Delage \cite{AmD15} 
give a comprehensive treatment of the topic from 
minimax preference robust optimization (PRO) perspective.  
Specifically,
}
they propose to use  available information of the decision maker's utility preference
such as preferring certain lotteries over other lotteries and being risk averse, $S$-shaped or prudent
to
construct an ambiguity
set of plausible utility functions
and then base the optimal decision
on the worst case utility function from the ambiguity set.
Hu and Mehrotra \cite{HuM15}  consider a probabilistic representation of the class of increasing concave utility functions by
confining them to a compact interval and normalizing
them with range $[0,1]$.
In doing so, they propose a
moment-type
framework for constructing
the ambiguity set of the decision maker's utility preference
which covers a number of important approaches such as
the certainty equivalent and pairwise comparison.
Hu and Stepanyan
\cite{HuS17} propose a so-called reference-based almost
stochastic dominance method for constructing
a set of utility functions
near a reference utility
which satisfies certain stochastic dominance relationship
and use the set to characterize the decision maker's preference.
{\color{black}
Over the past few years, the research on PRO has received increasing attentions
in the communities of stochastic/robust optimization and risk management, see
for instances \cite{HHX18,HXH22,DGX22,ZhX21,Li21,LCX21} and references therein.}

{\color{black}
In both (MOCE) and (RMOCE) models, the true probability distribution $P$ is assumed to be known. In the data driven problems, the true $P$ is 
unknown
but it is possible to use empirical data to construct an approximation of $P$.
    Unfortunately, such data may be contaminated and consequently we may be concerned
    by the quality of the MOCE values calculated as such. This kind of issue
    is well studied in robust statistics \cite{HuR09} and can be traced down to 
    earlier work of Hample \cite{Ham71}.
Cont et al.~\cite{CDS10} first study the quality of 
the plug-in estimators of 
law invariant risk measures using Hampel's classical concept of qualitative robustness \cite{Ham71}, 
that is, the plug-in estimator of a risk functional is said to be qualitatively robust 
if it is insensitive to the variation of sampling data. 
According to Hampel's theorem, Cont et al.~\cite{CDS10} 
demonstrate that the qualitative robustness of a 
plug-in estimator 
is equivalent to the weak continuity of the risk functional 
and 
that value at risk (VaR) is qualitatively robust whereas conditional value at risk (CVaR) is not.
Kr\"atschmer et al.~\cite{KSZ12} argue that the use of Hampel's classical concept of qualitative robustness may be problematic 
because it requires the risk measure essentially to be insensitive with respect to the tail behaviour of the random variable 
and the recent financial crisis shows that a faulty estimate of tail behaviour can lead to a drastic underestimation of the risk. 
Consequently, they propose a refined notion of qualitative robustness that applies 
also to tail-dependent statistical functionals and that allows 
one to compare statistical functionals in regards to their degree of robustness. 
The new concept captures the trade-off between robustness and sensitivity and can be quantified by an index of qualitative robustness. 
Guo and Xu \cite{GuX21SR} take a step forward by deriving quantitative statistical robustness of PRO models. Xu and Zhang \cite{XuZ21} extend the analysis to distributionally robust optimization models.
}

	In this paper, we consider a situation where the decision maker has a
	nominal utility function but is short of complete information as to whether
	it is the true. Consequently we propose to use the Kantorovich ball centered at the nominal utility function as
	the ambiguity set.
	We begin
	with piecewise linear utility (PLU) functions 
	defined over a convex and closed interval of $\R$
	and show that the inner minimization problem in the definition of RMOCE can be reformulated as a linear program when $\xi$ has a finite discrete distribution. We then
	propose an iterative algorithm to compute the RMOCE
	by solving the inner minimization problem and outer maximization problem alternatively.
	
	To extend the scope of the proposed computational method, we extend the discussion to the cases that the utility functions
	are not necessarily piecewise linear and the domain of the utility function is unbounded. We derive error bounds
	arising from using PLU-based RMOCE to approximate the general RMOCE. Since our numerical scheme for computing the RMOCE
	is based on the samples of $\xi$, we study statistical robustness
	of the sample-based RMOCE to address the case that
	the sample data of $\xi$ are potentially contaminated.
	Finally we carry out some numerical tests on the proposed computational schemes for concave utility functions.
	
The rest of the paper are organized as follows.
\Cref{sec:Properties} discusses 
the basic properties of MOCE and RMOCE. \Cref{sec:numer-methods} presents numerical schemes for computing the RMOCE when the utility functions in the ambiguity set are piecewise linear.
\Cref{sec:error} 
details approximation of the ambiguity set of general 
utility functions by the ambiguity set of piecewise linear utility functions and its effect on RMOCE.
\Cref{sec:unbounded}  discusses the RMOCE model with utility function having unbounded domain and  
streamlines the potential extensions of the MOCE model
 to multi-attribute decision making.
\Cref{sec:quantitative} discusses statistical robustness 
of RMOCE when it is calculated with
contaminated data. 
\Cref{sec:numerical} reports numerical results and 
finally 
\Cref{sec:conclu} concludes with a brief summary of the main contributions of the paper.

\section{
	Properties of MOCE and RMOCE}
\label{sec:Properties}
We begin by discussing the well-definedness of MOCE and RMOCE.
Let $L_p(\Omega,\F,\mathbb{P})$ denote the space of random variables mapping from $(\Omega,\F,\mathbb{P})$ to $\R$ with finite $p$-th order moments and $\xi\in L_p(\Omega,\F,\mathbb{P})$.  Let $\mathscr{U}:\R\to \R$ be
the set of nondecreasing concave utility functions.
Throughout this paper, we make a blanket assumption to ensure the
well-definedness of the expected utility in the definitions of MOCE and RMOCE.
\begin{assumption}
	\label{A:growth-phi-x-CE}
	There exist gauge functions $\phi_1:\R\to\R$ and
	$\phi_2:\R\to\R$ parameterized by $x$ satisfying 
	$\bbe_P[\phi_i(\xi)]< \infty$ for $i=1,2$ 
	such that 
	$$
	|u(\xi-x)|\leq \phi_1(\xi)  
	\quad\text{and}\quad
	\sup_{u\in {\mathscr U}} |u(\xi-x)|\leq \phi_2(\xi), \forall x, \xi\in\R.
	$$
\end{assumption}
The condition stipulates the interaction between the tails
distribution of $\xi$ and tails of the utility function.
We refer readers to Guo and Xu \cite{GuX21} for more detailed discussions on this. To facilitate the forthcoming
discussions, we let $\mathscr{P}(\Xi)$ denote the set of probability measures on $\Xi\subset\R$, and for each fixed $x$, define
$$
{\cal M}^{\phi_i} := \{
P\in \mathscr{P}(\Xi): \bbe_P[\phi_i(\xi)]<\infty
\}
$$
for $i=1,2$.
Let ${\cal C}_\Xi^{\phi_i}$ denote the class of continuous functions $h:\Xi\to\R$ such that $|h(t)|\leq C(\phi_i(t)+1)$ for all $t\in\Xi$.
The $\phi_i$-topology, denoted by $\tau_{\phi_i}$, is the coarsest topology  on ${\cal M}^{\phi_i}$ for which
the mapping
$g_h:=\int_{\Xi} h(z) P(dz),\; h\in {\cal C}_\R^{\phi_i}
$
is continuous.
A sequence  $\{P_N\} \subset {\cal M}^{\phi_i}$
is said to converge $\phi_i$-weakly to $P\in {\cal M}^{\phi_i}$ written
${P_N} \xrightarrow[]{\phi_i} P$ if it converges
w.r.t.~$\tau_{\phi_i}$. Note that in the case that when the support set of $\xi$ is a  compact set in $\R$, then the $\phi_i$-topology reduces to ordinary topology of weak convergence.

	

Our first technical result is on the attainability of
the optimum in the definition of MOCE.

\begin{prop}
	\label{Prop:Attain-max-RMOCE}
	Assume: (a)
	\Cref{A:growth-phi-x-CE}
	holds,
	(b) there exists $\alpha$ such that
	$\{x\in\R: u(x)+\bbe_P[u(\xi-x)]\geq \alpha\}$ is a compact set,
	(c) the support set of $\xi$, denoted by $\Xi=[\xi_{\min},\xi_{\max}]$, is bounded,
	(d) $u$ is strictly concave over $\Xi$.
	Then for $P\in {\cal M}^{\phi_1}$,
	\begin{equation}
	\label{eq:M_u-opti}
	M_u(\xi)=\sup_{x\in[\xi_{\min}/2,\;
		\xi_{\max}/2]}\{u(x)+\bbe_P[u(\xi-x)]\}.
	\end{equation}
	Moreover, if $\{P_N\}\subset \mathscr{P}(\Xi)$
	and $P_N$ converges weakly
	to $\delta_{\hat{\xi}}$, the Dirac probability measure at $\hat{\xi}$, then
	$M_u(\xi_N)$ converges to $2u(\hat{\xi}/2)$.
\end{prop}

\begin{proof}
	Since $u$ is a strictly concave function,
	(\ref{eq:OCE-new}) is a convex optimization problem.
	Condition (b) ensures existence of an optimal solution, denoted by $x^*$. Following a similar analysis to the proof of \cite[Lemma 2.1]{BTT07}, we can write down the first order optimality condition of the program at $x^*$,
	\begin{equation}
	0\in\partial u(x^*)+\partial\bbe_P[u(\xi-x^*)],
	\label{eq:first-order-opti}
	\end{equation}
	where $\partial u$ denotes convex subdifferential \cite{Roc70}.
	Since  $\partial u(x)=[u_+'(x),u_-'(x)]$ for any $x\in \R$, where
	$u'_-, u'_+$ denote the left derivative and right derivative of $u$ at $x$ and
	$$
	\partial\bbe_P[u(\xi-x^*)] = - \bbe_P[\partial u(\xi-x^*)],
	$$
	where the expectation/integration at the right hand side is in the sense of Aumann \cite{Aum65}. Consequently we can rewrite
	(\ref{eq:first-order-opti}) as
	\begin{eqnarray}
	0&\in &[u_+'(x^*),u_-'(x^*)]-\bbe_P\left[[u_+'(\xi-x^*),u_-'(\xi-x^*)]\right]\nonumber\\
	&=&[u_+'(x^*),u_-'(x^*)]-\left[\bbe_P[u_+'(\xi-x^*)],\bbe_P[u_-'(\xi-x^*)]\right],
	\label{eq:first-order-opti-1}
	\end{eqnarray}
	which yields
	$$
	u'_+(x^*)-\bbe_P[u'_-(\xi-x^*)]\leq 0\leq u'_-(x^*)-\bbe_P[ u'_+(\xi-x^*)].
	$$
	Since $u'_-$ and $u'_+$
	are non-increasing,
	the inequality above implies
	\begin{equation}
	\label{eq:first-order-opti-1-a}
	u'_+(x^*)\leq \bbe_P[u'_-(\xi-x^*)]\leq u'_-(\xi_{\min}-x^*)
	\end{equation}
	and
	\begin{equation}
	\label{eq:first-order-opti-1-b}
	u'_-(x^*)\geq\bbe_P[u'_+(\xi-x^*)]\geq u'_+(\xi_{\max}-x^*).
	\end{equation}
	Moreover, since  $u_-'(t')> u_+'(t'')$ for any $t'<t''$, then
	inequalities (\ref{eq:first-order-opti-1-a})-(\ref{eq:first-order-opti-1-b}) imply
	$$
	x^*\geq \xi_{\min}-x^* \quad {\rm and} \quad  x^*\leq \xi_{\max}-x^*,
	$$
	and hence (\ref{eq:M_u-opti}).
	
	The second part of the claim follows directly from the first part
	in that the interval $[\xi_{\min}^N/2, \\
	\xi_{\max}^N/2]$ converges to a single point $\hat{\xi}/2$ and
	$\tau_{\phi_1}$-convergence coincides with the weak convergence
	because of the restriction of the range of $\xi$ to a compact subset of $\R$.
	\hfill $\Box$
\end{proof}

Ben-Tal and Teboulle \cite{BTT07} derive
a similar result to the first part of the proposition
for the optimized certainty equivalent and demonstrate that
$S_u(\xi)\in [\xi_{\min},\xi_{\max}]$ under the conditions that
$u$ is concave  and $1\in \partial u(0)$ rather than strictly concave.
{\color{black} Strict concavity is needed
	to ensure the optimum in (\ref{eq:M_u-opti}) to be achieved
	in $[\xi_{\min}/2,\xi_{\max}/2]$. We can find a counter example
	otherwise, see \Cref{ex:PWL-utility}.

	Like the optimized certainty equivalent, the newly defined modified optimized
	certainty equivalent enjoys a number of properties as stated in the next proposition.

	\begin{prop}[Properties of MOCE]
		\label{prop:property-MOCE}
		Let $u:\R\rightarrow(-\infty,+\infty)$ be a closed proper function. Under \Cref{A:growth-phi-x-CE}, the following assertions hold.
		\begin{itemize}
			\item[(i)] $M_u$ is law invariant.
			
			\item[(ii)](Monotonicity)
			For any $\xi_1\leq \xi_2\in L_p(\Omega,\F,\mathbb{P})$,
			with respective distributions (push-forward probabilities)
			$P_1,P_2\in {\cal M}^{\phi_1}$,
			$M_u(\xi_1)\leq M_u(\xi_2)$.
			
			\item[(iii)](Risk aversion)
			If $u(t)\leq t$ for all $t\in\R$, then $M_u(\xi)\leq\bbe_P[\xi]$ for any random variable $\xi$.
			
			\item[(iv)](Second-order stochastic dominance)
			Let $\xi_1, \xi_2$ be random variables with compact support.
			Then for any concave utility function $u$,
			$$
			M_u(\xi_1)\geq M_u(\xi_2)\Longleftrightarrow C_u(\xi_1)\geq C_u(\xi_2),
			$$
			where $C_u(\xi):=u^{-1}(\bbe_P[u(\xi)])$ is the classical
			certainty equivalent.
			
			\item[(v)](Concavity and positive subhomogeneity)
			If $u$ is concave, then $M_u(\cdot)$ is also concave. Moreover, if $u(0)\geq0$, then
			\begin{equation}
			\label{eq:Mu-sub-hom}
			M_u(\delta \xi)\leq\delta M_u(\xi), \;\forall \delta\in [1, \infty)
			\quad {\text{and}} \quad
			M_u(\delta \xi)\geq\delta M_u(\xi), \;\forall \delta\in [0,1].
			\end{equation}

		\end{itemize}
		
	\end{prop}

	\begin{proof} 
		Parts (i)-(iii) follow straightforwardly from the definitions, we prove the rest.
		
		Part (iv).
		``$\Longleftarrow$''.
		By the definition of certainty equivalent, $C_u(\xi_1)\geq C_u(\xi_2)$ implies
		$\bbe_P[u(\xi_1)]\geq\bbe_P[u(\xi_2)]$ for all concave utility functions.
		The latter implies  $\xi_1$ dominates $\xi_2$ in second order, which in turn
		guarantees
		$\xi_1-x$ dominates $\xi_2-x$ in second order for any fixed $x\in\R$.
		Consequently
		$\bbe_P[u(\xi_1-x)]\geq\bbe_P[u(\xi_2-x)]$ for any $x\in\R$.
		Adding both sides of the inequality by $u(x)$ and taking the maximum, we obtain  $M_u(\xi_1)\geq M_u(\xi_2)$.
		
		``$\Longrightarrow$''.
		Let $x_1, x_2$ be the points where the supremum of $M_u(\xi_1)$ and $M_u(\xi_2)$ are
		attained. Then
		\begin{eqnarray*}
		M_u(\xi_1) = u(x_1)+\bbe_P[u(\xi_1-x_1)] &\geq& M_u(\xi_2) =u(x_2)+\bbe_P[u(\xi_2-x_2)] \\
		&\geq& u(x_1)+\bbe_P[u(\xi_2-x_1)],
		\end{eqnarray*}
		which yields
		$\bbe_P[u(\xi_1-x_1)]\geq\bbe_P[u(\xi_2-x_1)]$.
		The latter implies $\bbe_P[u(\xi_1)]\geq\bbe_P[u(\xi_2)]$.

		Part (v).
		First we prove the concavity of $M_u$, i.e. for $\lambda\in(0,1)$ and any random variables $\xi_1$, $\xi_2$,
		$$
		M_u(\lambda\xi_1+(1-\lambda)\xi_2))\geq\lambda M_u(\xi_1)+(1-\lambda)M_u(\xi_2).
		$$
		Since $u$ is concave, the function $f(z,x):=u(x)+u(z-x)$ is joint concave over $\R\times\R$. Therefore, for any $x_1, x_2\in\R$, with $x_{\lambda}:=\lambda x_1+(1-\lambda) x_2$ and $\xi_{\lambda}:=\lambda \xi_1+(1-\lambda) \xi_2$, one has
		$$
		\bbe[f(\xi_{\lambda},x_{\lambda})]\geq\lambda\bbe[f(\xi_1,x_1)]+(1-\lambda)\bbe_P[f(\xi_2,x_2)].
		$$
		Since $M_u(\xi_{\lambda})=M_u(\lambda\xi_1+(1-\lambda)\xi_2))=\sup_{x\in\R}\bbe_P[f(\xi_{\lambda},x)]$, it follows that
		\begin{eqnarray}
		M_u(\xi_{\lambda}) &\geq& \sup_{x_1,x_2}\left\{\lambda\bbe_P[f(\xi_1,x_1)]+(1-\lambda)\bbe_P[f(\xi_2,x_2)]\right\} =
		\lambda M_u(\xi_1)+(1-\lambda)M_u(\xi_2). \nonumber
		\end{eqnarray}
		Next, we turn to prove the subhomogeneity of $M_u$.
		Let $s(\delta):=\frac{1}{\delta}M_u(\delta\xi)$, for $\delta>0$. Then
		\begin{equation}
		s(\delta)=\sup_{x\in\R}\left\{\frac{1}{\delta}u(\delta x)+\bbe_P\left[\frac{1}{\delta}u(\delta(\xi-
x))\right]\right\}.
		\label{eq-subhomogeneous}
		\end{equation}
		Let $\delta_2>\delta_1>0$.
		By the concavity of $u$,
		$$\frac{u(\delta_2 t)-u(0)
		}{\delta_2-
			0}\leq\frac{u(\delta_1 t)-u(0)}{\delta_1-0}.$$
		Since $u(0)\geq0$, the above inequality implies
		\begin{equation}
		\frac{1}{\delta_2}u(\delta_2 t)\leq\frac{1}{\delta_1}u(\delta_1 t), \; \forall t\in\R.
		\label{eq-subhomogeneous_monotonicity}
		\end{equation}
		Inequality  (\ref{eq-subhomogeneous_monotonicity}) also implies
		\begin{equation*}
				\bbe_P\left[\frac{1}{\delta_2}u(\delta_2(\xi-x))\right]
		\leq \bbe_P\left[\frac{1}{\delta_1}u(\delta_1(\xi-x))\right].
		\end{equation*}
		A combination of the two inequalities implies the objective function in  (\ref{eq-subhomogeneous}) is non-increasing in $\delta$ and hence $s(\delta)$.
		By setting $\delta_1$ and $\delta_2$ to $1$ respectively
		in the inequality above, we obtain (\ref{eq:Mu-sub-hom}).
		\hfill $\Box$
	\end{proof}

	Next, we discuss
	how the utility function $u$
	may be recovered
	from a given
	modified certainty equivalent $M_u(\xi)$, which is an important
	property enjoyed by the OCE.
	Let
	$$\xi_p=\left\{
	\begin{array}{ll}
		z & \text{with probability} \; p,\\
		0 & \text{with probability} \; 1-p,
	\end{array}
	\right.$$
	where $0<p<1$ and $z>0$.
	For a
	concave utility function $u$,
	the modified optimized certainty equivalent $M_u(\xi_p)$ can be
	written as
	\begin{equation}
	M_u[z,p]:=\sup_{0\leq x\leq z/2}\{u(x)+pu(z-x)+(1-p)u(-x)\}.
	\label{eq:Muxp}
	\end{equation}
	

	\begin{prop}
	\label{prop-restore}
		If $u$ is a strong risk averse utility function, i.e., $u(t)<t$ for all $t\neq0$, and $u(0)=0$, then
		$\lim_{p\rightarrow0^+}\frac{M_u[z,p]}{p}=u(z).$
	\end{prop}

	\begin{proof}
		Observe that $x^*=0$ is the optimal solution of problem (\ref{eq:Muxp})
		if and only if
		{\color{black}
			\begin{eqnarray}
			u(x)+pu(z-x)+(1-p)u(-x) &\leq&
			u(0)+pu(z-0)+(1-p)u(-0)\nonumber\\
			&=&pu(z),   \forall x\in [0,z/2].
			\label{eq:Muxp-1}
			\end{eqnarray}
			The inequality above
			can be equivalently written as
			\begin{equation}
			p[u(z-x)-u(-x)-u(z)]\leq -u(-x)-u(x), \forall x\in [0,z/2].
			\label{eq:Muxp-2}
			\end{equation}
		}
		Since $u$ is strongly risk averse, that is, $u(t)<t$ for all $t\neq0$, then $-u(-x)-u(x)>0$ and hence inequality (\ref{eq:Muxp-2}) holds for $p$ sufficiently small. This in turn
		shows that inequality (\ref{eq:Muxp-2}) holds and hence
		$x^*=0$ is the optimal solution of problem (\ref{eq:Muxp})
		for all $p$ sufficiently small. Thus we have
		$M_u[z,p]=pu(z)$
		and the conclusion follows.
		\hfill $\Box$
	\end{proof}
}

\begin{ex}
	We give a few examples which illustrate how MOCE can be calculated in a closed form and their difference in comparison with OCE.
	Let $M_u(\xi)=2(1-\left(\bbe_P[e^{-\xi}]\right)^{1/2})$. Then
	$M_u[z,p]=2-2(pe^{-z}+(1-p))^{1/2}$
	and
	\begin{eqnarray*}
		u(z) = \lim_{p\rightarrow0^+}\frac{M_u[z,p]}{p}= \lim_{p\rightarrow0^+}\frac{2-2(pe^{-z}+1-p)^{1/2}}{p}= \lim_{p\rightarrow0^+}\frac{1-e^{-z}}{(pe^{-z}+1-p)^{1/2}}.
	\end{eqnarray*}
	Hence, the recovered utility function is $u(z)=1-e^{-z}$.
\end{ex}

\begin{ex}[Exponential Utility Function]
	Let $u(t)=1-e^{-t}$, $t\in\R$.
	It is easy to derive that
	the
	optimal solution of problem (\ref{eq:OCE-new}) is
	$x^*=-\frac{1}{2}\ln \bbe_P\left[e^{-\xi}\right]$
	and the 
	modified optimized certainty equivalent is
	$
	M_u(\xi)=2\left(1-\left(\bbe_P[e^{-\xi}]\right)^{1/2}\right).
	$
	On the other hand, it follows from \cite{BTT07} that
	$
	S_u(\xi)=-\ln\bbe_P[e^{-\xi}].
	$
	Since $u(t) \leq t$ for all $t\in \R$, then we can deduce
	from the definitions that $M_u(\xi)\leq S_u(\xi)$. Indeed the strict inequality holds in that $u(t)=t$ only at $t=0$.
\end{ex}



\begin{ex}[Piecewise Linear Utility Function]
	\label{ex:PWL-utility}
	Let
	$$
	u(t)=\left\{\begin{array}{ll}
		\gamma_2t & \text{if} \;\; t\leq0, \\
		\gamma_1t & \text{if} \;\; t>0,
	\end{array}
	\right.$$
	where $0\leq\gamma_1<1\leq\gamma_2$. Then the utility function $u$ can be written as
	$u(t)=\gamma_1(t)_+-\gamma_2(-t)_+$
    and the modified optimized  certainty equivalent is
	\begin{equation}
	\label{eq:M_u-PWL}
	M_u(\xi)=\sup_{x\in\R}\{\gamma_1(x)_+-\gamma_2(-x)_+-\gamma_2\bbe_P[(x-\xi)_+]+\gamma_1\bbe_P[(\xi-x)_+]\}.
	\end{equation}
	Compared to optimized certainty equivalent (see \cite{BTT07})
	$$
	S_u(\xi)=\sup_{x\in\R}\{x-\gamma_2\bbe_P[(x-\xi)_+]+\gamma_1\bbe_P[(\xi-x)_+]\},
	$$
	we can also conclude that  $M_u(\xi)\leq S_u(\xi)$
	because  $u(t)\leq t$ for all $t\in\R$.
	
	{\color{black}
		It  might be interesting to see where the optimum
		in (\ref{eq:M_u-PWL}) is achieved. We consider the case that
		$P$ follows a Dirac distribution at point $t>0$, that is,
		$[\xi_{\min},\xi_{\max}]=\{t\}$.
		Consequently
		\begin{eqnarray}
		u(x)+\bbe_P[u(\xi-x)]=\left\{
		\begin{array}{ll}
			(\gamma_2-\gamma_1)x+t\gamma_1 & \; {\rm if} \; x\leq0, \\
			t\gamma_1 & \; {\rm if} \; 0<x\leq t, \\
			(\gamma_1-\gamma_2)x+t\gamma_2 & \; {\rm if} \; x\geq t.
		\end{array}
		\right.
		\end{eqnarray}
		The set of optimal solutions is $[0,t]$,
		which is not contained in $[0,t/2]\not\subset
		[\xi_{\min}/2,  \xi_{\max}/2]=\{t/2\}$.
		This explains that (\ref{eq:M_u-opti}) may fail to hold
		without strict concavity of $u$.}

\end{ex}

We now move on to discuss the properties of the robust modified optimized certainty equivalent.

\begin{prop}[Properties of RMOCE] 
\label{prop-RMOCE}
Let $u:\R\rightarrow[-\infty,+\infty)$ be a closed proper function. Under \Cref{A:growth-phi-x-CE}, the following assertions hold.
	\begin{itemize}
		\item[(i)] $R(\xi)$ is law invariant.
		
		\item[(ii)](Monotonicity)
		For any $\xi_1\leq \xi_2\in L_p(\Omega,\F,\mathbb{P})$,
		with respective distributions (push-forward probabilities)
		$P_1,P_2\in {\cal M}^{\phi_2}$,
		$R(\xi_1)\leq R(\xi_2)$.
		
		\item[(iii)](Risk aversion)
		If $u(t)\leq t$, for all $t\in\R$ and $u\in\mathscr U$, then $R(\xi)\leq\bbe_P[\xi]$, for any random variable $\xi$.
		
		\item[(iv)](Second-order stochastic dominance)
		Let $\xi_1, \xi_2$ be random variables with compact support. Then for any concave utility function $u$,
		$$
		C_u(\xi_1)\geq C_u(\xi_2)\Longrightarrow R(\xi_1)\geq R(\xi_2),
		$$
		where $C_u(\xi)$ is the classical certainty equivalent.
		
		\item[(v)](Concavity and positive subhomogeneity)
		If $u$ is concave, then $R(\cdot)$ is also concave. Moreover, if $u(0)\geq0$, then
		\begin{equation}
		R(\delta \xi)\leq\delta R(\xi), \;\forall \delta\in [1, \infty)
		\quad {\text{and}} \quad
		R(\delta \xi)\geq\delta R(\xi), \;\forall \delta\in [0,1].
		\end{equation}
		
	\end{itemize}
	
\end{prop}

\begin{proof} Parts (i)-(iii) are obvious.
	
	Part (iv). Following a similar argument to the proof of part (iv)
of 	\Cref{prop:property-MOCE}, we can show that
	$C_u(\xi_1)\geq C_u(\xi_2)$ implies
	$\bbe_P[u(\xi_1)]\geq\bbe_P[u(\xi_2)]$ and
	$\bbe_P[u(\xi_1-x)]\geq\bbe_P[u(\xi_2-x)]$  for any fixed
	$x\in\R$ and hence
	$$
	u(x)+\bbe_P[u(\xi_1-x)]\geq
	u(x) + \bbe_P[u(\xi_2-x)].
	$$
	Taking infimum on both sides w.r.t. $u$ over ${\cal U}$ and
	then supremum w.r.t. $x$, we obtain $R(\xi_1)\geq R(\xi_2)$.
	
	Part (v). Let
	$g_u(\delta,x) :=
	\frac{1}{\delta}u(\delta x)+\bbe_P\left[\frac{1}{\delta}u(\delta(\xi-x))
	\right].
	$
	We can show as in the proof of \Cref{prop:property-MOCE} (v) that $g_u(\cdot,x)$ is non-increasing
	over $\R$. This property is preserved after taking the infimum in $u$ over ${\mathcal U}$ and then supremum in $x$ over $\R$.
	\hfill $\Box$
\end{proof}

{\color{black}
Before concluding this section, we remark that 
it is possible to use 
a different utility function $v$ for 
the present consumption $x$, i.e.,
\begin{equation}
{\rm (RMOCE')}\quad \quad  \displaystyle{M_{u,v}(\xi):=\sup_{x\in\R} \; \{v(x)+\bbe_P[u(\xi-x)]\}}.
\label{eq:OCE-new-v}
\end{equation}
In that case, 
some of the 
properties of MOCE may be retained. 
For example, 
law invariance, monotonicity, risk aversion, concavity, positive subhomogeneity 
and 
second-order stochastic dominance
are all satisfied when $v$ enjoys the same property as $u$.
\Cref{prop-restore}
also holds when $v$ satisfies the same property as $u$.
However, the change will have an effect on \Cref{Prop:Attain-max-RMOCE},
in which case it will be difficult to estimate 
the interval containing the optimal solution.





}


\section{
	Computation of RMOCE}
\label{sec:numer-methods}

%

Having
investigated the properties of MOCE and RMOCE in the previous section, we move on to discuss numerical schemes for computing RMOCE in this section. To this end, we need to have a concrete structure of the ambiguity set. As reviewed in the introduction, various approaches have been proposed for constructing an ambiguity set of utility functions in the literature of preference robust optimization depending on the availability of information. Here we consider a situation where
the decision maker has a nominal utility function obtained from empirical data or subjective judgement but lacks of complete information to identify whether it is the true utility function which captures precisely the decision maker's preference.
Consequently we may construct a ball of utility functions centered at the nominal utility function under some appropriate metrics.
Here we concentrate on the Kantorovich metric.

\subsection{Kantorovich ball of piecewise linear utility functions}

We begin by considering a ball of utility function centered 
at a piecewise linear utility function under the the Kantorovich metric. 
In practice, decision maker's utility preferences
are often elicited through questionnaires. 
For example, a customer's utility preference 
may be elicited via the customer's willingness to pay at certain price points \cite{DeJ19,LCX21}. From computational point of view,
piecewise linear utility function may bring significant 
convenience to calculation of OCE, see
Nouiehed et al.~\cite{NPR19}.

Let $t_1<\cdots<t_N$ be an ordered sequence of points in $[a,b]$
and $T:=\{t_1,\cdots,t_{N}\}$ with $t_1=a$ and $t_N=b$.
Let $\mathscr{U}_N$
be a class of continuous, non-decreasing, concave, piecewise linear functions
defined over an interval $[a,b]$ with kinks on $T$, as well as Lipschitz condition with modulus $L$ and normalized conditions $u(a)=0$ and $u(b)=1$.
Let $u_N,u_N^0\in \mathscr{U}_N$, we consider a ball in $\mathscr{U}_N$
with the Kantorovich metric
\begin{equation}
\mathbb{B}_{K}(u_N^0,r) = \left\{u_N\in \mathscr{U}_N | \dd_K
(u_N,u_N^0)\leq r\right\},
\label{bl-ball-utl}
\end{equation}
where the subscript $K$ represents the Kantorovich metric and
\begin{equation}
\displaystyle{\dd_K (u,v):= \sup_{g\in\mathscr G_K}|\langle g,u\rangle-\langle g,v\rangle|=\sup_{g\in\mathscr G_K}\left\{\int_{\R}g(t)du(t)-\int_{\R}g(t)dv(t)\right\}}
\label{eq:Kant-metric}
\end{equation}
and
\begin{equation}
\label{eq-kantorovich}
\mathscr G_K:= \{g:\R\rightarrow\R \mid |g(t)-g(t')|\leq|t-t'|, \forall t,t'\in\R\}.
\end{equation}
{\color{black}
Note that piecewise linear 
utility functions are used to 
approximate general utility functions in the utility preference robust optimization model \cite{GuX21}.
The difference is that here we use
the Kantorovich ball to construct 
the ambiguity set of DM's utility function whereas the authors use pairwise comparison approach to elicit the DM's utility preferences in \cite{GuX21}.  
}
The next proposition states that $\dd_K(u_N,u_N^0)$ may be computed by solving a linear program.

\begin{prop} The Kantorvich distance
	$\dd_K(u_N,u_N^0)$ is the optimal value of the following linear program:
	\begin{subequations}
		\label{eq:Kant-u-v-LP}
		\begin{eqnarray}
			\hspace{-0.5cm}
			\displaystyle \max_{\substack{y_1,\cdots,y_{N-1} \\ z_1,\cdots,z_N}}
			&&\sum_{j=2}^N  (\beta_{j-1}-\beta_{j-1}^0)y_{j-1}\\
			{\rm s.t.}~~~
			&&  y_{j-1}\leq z_{j-1}(t_{j}- t_{j-1})+\frac{1}{2}(t_{j}- t_{j-1})^2,  j=2,\cdots,N, \\
			&&-y_{j-1}\leq -z_{j-1}(t_{j}- t_{j-1})+\frac{1}{2}(t_{j}- t_{j-1})^2, j=2,\cdots,N, \\
			&&  y_{j-1}\leq z_j(t_{j}- t_{j-1})+\frac{1}{2}(t_{j}- t_{j-1})^2,  j=2,\cdots,N, \\
			&& -y_{j-1}\leq -z_j(t_{j}- t_{j-1})+\frac{1}{2}(t_{j}- t_{j-1})^2, j=2,\cdots,N.
		\end{eqnarray}
	\end{subequations}
\end{prop}
\begin{proof}
	Let $g\in \mathscr{G}_K$. By definition,
	$$
	\int_a^b g(t)du_N(t) = \sum_{j=2}^N \beta_{j-1}\int_{t_{j-1}}^{t_j} g(t)dt,
	$$
	where $\beta_j$ denotes the slope of $u_N$ at interval
	$[t_{j-1}, t_j]$. Since for each $g\in  \mathscr{G}_K$,
	$-g\in  \mathscr{G}_K$,
	$$
	\dd_K(u_N,u_N^0) = \sup_{g\in \mathscr{G}_K}\sum_{j=2}^N (\beta_{j-1}-\beta_{j-1}^0)\int_{t_{j-1}}^{t_j} g(t)dt,
	$$
	where $\beta_{j-1}^0$ denotes the slope of $u^0$ at interval
	$[t_{j-1}, t_j]$. Note that in this formulation,
	$\dd_K(u_N,u^0_N)$ depends on the slopes of $u_N,u_N^0$ rather than their function values.
	Let $y_{j-1} :=\int_{t_{j-1}}^{t_j} g(t)dt$ and
	$z_j:=g(t_j)$. Since
	$|g(t)-g(t_{j-1})|\leq t-t_{j-1}$ for all $t\in [t_{j-1},t_j]$,
	we have
	$$
	z_{j-1}(t_{j}- t_{j-1})-\frac{1}{2}(t_{j}- t_{j-1})^2 \leq  y_{j-1}\leq z_{j-1}(t_{j}- t_{j-1})+\frac{1}{2}(t_{j}- t_{j-1})^2
	$$
	for $j=2,\cdots,N$. 
	Likewise,
	since
	$|g(t)-g(t_{j})|\leq t_{j}-t$ for all $t\in [t_{j-1},t_j]$,
	we have
	$$
	z_{j}(t_{j}- t_{j-1})-\frac{1}{2}(t_{j}- t_{j-1})^2 \leq  y_{j-1}\leq z_{j}(t_{j}- t_{j-1})+\frac{1}{2}(t_{j}- t_{j-1})^2
	$$
	for $j=2,\cdots,N$.
	To complete the proof, it suffices to show that conditions
	\begin{equation}
	|g(t)-g(t_{j-1})|\leq t-t_{j-1}
	\;\; \text{and} \;\;
	|g(t)-g(t_{j})|\leq t_{j}-t, \forall t\in [t_{j-1},t_j]
	\label{eq:g-calm-lft-rgt}
	\end{equation}
	are adequate to cover the generic condition
	\begin{equation}
	\label{eq:Lip-Kant}
	|g(t')-g(t'')| \leq |t'-t''|, \forall t',t''\in [a,b].
	\end{equation}
	We consider two cases.
	
	Case 1. $t', t'' \in [t_{i-1}, t_i]$ for some $i$.
	In this case, the generic condition is
	adequately covered by $|g(t)-g(t_{j-1})|\leq t-t_{j-1}$ for all $t\in[t_{j-1},t_j]$.
	Because the objective depends only
	on $\int_{t_{j-1}}^{t_j} g(t)dt$.
	
	Case 2. $t', t''$ lie in two  intervals, i.e., $t'\in [t_{i-1},t_{i}]$ and  $t''\in [t_{j-1},t_{j}]$, where $i< j$.
	Then by (\ref{eq:g-calm-lft-rgt}), 
	\begin{eqnarray*}
	|g(t')-g(t'')| &\leq&
	|g(t')-g(t_{i})|+ |g(t_{i})-g(t_{i+1})|+\cdots +|g(t_{j-2})-g(t_{j-1})|+|g(t_{j-1})-g(t'')|  \\
	&\leq& t_i-t' + t_{i+1}-t_{i} + \cdots+t_{j-1}-t_{j-2}+t''-t_{j-1}= t''-t'.
	\end{eqnarray*}
	The proof is complete.
	\hfill $\Box$
\end{proof}

\subsection{Alternating iterative algorithm for
	computing RMOCE
}
\label{subsec:numer-methods}

We are now ready to discuss how to compute the RMOCE with the ambiguity set of
piecewise linear utility functions constructed by the Kantorovich ball.
Assume 
that the probability distribution of random variable $\xi$ is discrete with $P(\xi=\xi_k)=p_k$ for $k=1,...,K$ and $u_N, u_N^0\in\mathscr U_N$.
Then we can rewrite the RMOCE problem (\ref{eq:OCE-new-robust}) as
\begin{equation}
{\rm (RMOCE-PLU)}  \quad \;\;
R_N(\xi):=\displaystyle{\max_{x\in \R} \min_{u_N\in \mathbb{B}_K(u^0_N,r)}  }  \;\;  u_N(x)+\sum_{k=1}^K p_k u_N(\xi_k-x).
\label{eq-RCE-zeta-N-discrete}
\end{equation}
Recall that in \Cref{Prop:Attain-max-RMOCE},
we show that the optimal solutions of MOCE are contained in interval $[\xi_{\min}/2,\xi_{\max}/2]$
when utility function is strictly concave. Unfortunately, this result is not applicable to  problem (\ref{eq-nonconcave-subproblem2})
because $u_N^s$ is piecewise linear.  However,
under some fairly moderate conditions, we are able to
show that the optimal solutions are bounded.
The next proposition states this.

\begin{prop}
	\label{prop-piecewiselinear-optimality}
	Consider MOCE problem (\ref{eq:OCE-new}).
	Let $X^*$ denote the set of optimal solutions.
	Assume: (a) $u$ is a piecewise linear concave function
	and (b) $u$ has at least two pieces in the interval $[\xi_{\min},\xi_{\max}]$.
	Then  the following assertions hold.
	\begin{itemize}
		\item[(i)]  $X^*$ is a compact and convex set.
		
		\item[(ii)] If $0\in [\xi_{\min},\xi_{\max}]$, then
		$X^*\subset [\xi_{\min},\xi_{\max}]$.
		
		\item[(iii)] If $\xi_{\min}\geq 0$, then
		$X^*\subset [0,\xi_{\max}]$.
		
		\item[(iv)] If $\xi_{\max}\leq 0$, then
		$X^*\subset [\xi_{\min},0]$.
		
	\end{itemize}
\end{prop}

\begin{proof}
Part (i). Observe first that $X^*$ is a convex set since
problem (\ref{eq:OCE-new}) is a convex optimization problem. Suppose for the sake of a contradiction that $X^*$ is unbounded. Then either $X^*$ is a right half line or a left half line. We consider the former.
In that case, there exists $x^*\in X^*$ sufficiently large such that
\begin{equation}
[\xi_{\min},\xi_{\max}]\subset [\xi_{\min}-x^*, x^*].
\label{eq:PLU-INCLUSION}
\end{equation}
By the first order optimality condition
\begin{equation}
\label{eq-optimalsol-piecewiselinear}
0\in\partial u(x^*)+\partial \bbe_P[u(\xi-x^*)]=
\partial u(x^*)-\bbe_P[\partial u(\xi-x^*)].
\end{equation}
The equality holds because of Clarke regularity, see
\cite{BCS20,Cla90}.
Since $x^* \geq \xi_{\max}-x^*$, then any subgradient
in set $\partial u(x^*)$ is greater or equal to
the subgradient from $\partial u(\xi-x^*)$ for all $\xi\in [\xi_{\min},\xi_{\max}]$.
This means the optimality condition holds if and only if
$x^*$ and $\xi_{\min}-x^*$ are in the domain of the same linear piece. But this contradicts assumption (b).
Using a similar argument, we can also show that $X^*$ cannot be
a left half line.

Part (ii). Assume for a contradiction that $x^*> \xi_{\max}$.
Then inclusion (\ref{eq:PLU-INCLUSION}) holds.
Following a similar analysis to that in Part (i), we can show that in this case $x^*$ does not satisfy (\ref{eq-optimalsol-piecewiselinear}).
If $x^*<\xi_{\min}\leq0$, then
\begin{equation}
[\xi_{\min},\xi_{\max}]\subset [x^*, \xi_{\max}-x^*].
\label{eq:PLU-INCLUSION-2}
\end{equation}
Consequently we can show that $X^*$ cannot satisfy the optimality condition (\ref{eq-optimalsol-piecewiselinear}).

Part (iii). In this case, we can show that  $x^*$ cannot be larger that  $\xi_{\max}$ because otherwise we would have (\ref{eq:PLU-INCLUSION}) and a contradiction to the optimality condition.
Likewise if $x^*<0$, then the inclusion (\ref{eq:PLU-INCLUSION-2}) would be invoked.

Part (iv) is similar to Part (iii), we omit the details. 
\hfill $\Box$
\end{proof}

Note that  if we strengthen the condition on
two linear pieces in the interval $[\xi_{\min},\xi_{\max}]$
to a smaller interval $[\xi_{\min}/2,\xi_{\max}/2]$,
then we will be able to strengthen the conclusions
in Parts (ii)-(iv) whereby $X^*$ is included in $[\xi_{\min}/2,\xi_{\max}/2]$, we leave readers for an exercise.

Now we propose
the alternating iterative algorithm
for solving the maximin problem (\ref{eq-RCE-zeta-N-discrete}).
\begin{algorithm}
\label{alg-nonconcave}
	Step 0. Choose an initial point $x^0$.
	
	Step 1. For s=1,..., solve
	\begin{equation}\label{eq-nonconcave-subproblem1}
	\displaystyle{u^s_N \in \arg \min_{u_N\in \mathbb{B}_K(u^0_N,r)}} u_N(x^{s-1})+
	\sum_{k=1}^K p_k u_N(\xi_k-x^{s-1})
	\end{equation}
	and
	\begin{equation}\label{eq-nonconcave-subproblem2}
	\displaystyle{x^s\in \arg \max_{x\in X} u_N^s(x)+  \sum_{k=1}^K p_k u_N^s(\xi_k-x),
	}
	\end{equation}
	where $X$ is a compact subset of $\R$.
	
	Step 2. Stop when $x^{s+1}=x^s$ and $u_N^{s+1}=u_N^s$.
\end{algorithm}

{\color{black} 
Note that in equation (\ref{eq-nonconcave-subproblem2}), we 
restrict 
$x$ to
taking values in a convex and compact set $X$ since 
 \Cref{prop-piecewiselinear-optimality} guarantees
 that the optimal $x^*$ is contained in such a set. 
There is another important issue concerning the algorithm, that is, whether a sequence $\{x^s\}$ generated by the algorithm converges to the optimal solution of (RMOCE-PLU). 
The next proposition addresses this. 

}

{\color{black}
\begin{prop}
	\Cref{alg-nonconcave} either terminates in a finite number of steps with a solution of the (RMOCE-PLU) model or generates a sequence $\{(x^s,u_N^s)\}$ whose cluster points, if exist, are optimal solution of the (RMOCE-PLU) model.
\end{prop}
}
\begin{proof}
    Let $(x^*,u^*)$ be a cluster point of the sequence generated by Algorithm 3.1. Then for all $\mathbb{B}_K(u^0_N,r)$ and $x\in X$,
    \begin{equation}
        \label{eq-saddle}
        u^*(x)+\bbe_P[u^*(\xi-x)]\leq u^*(x^*)+\bbe_P[u^*(\xi-x^*)]\leq
    u(x^*)+\bbe_P[u(\xi-x^*)].
    \end{equation}
    For $s=1,2,...$,
    $$
    u^{s+1}(x^s)+\bbe_P[u^{s+1}(\xi-x^s)]\leq u(x^s)+\bbe_P[u(\xi-x^s)]
    $$
    and 
    \begin{equation}
        \label{eq-saddle1}
        u^s(x^s)+\bbe_P[u^s(\xi-x^s)]\leq u^s(x)+\bbe_P[u^s(\xi-x)].
    \end{equation}
    If Algorithm 3.1 terminates in finite steps, then $x^{s+1}=x^s$ and $u^{s+1}=u^s$ for some $s$ and $(x^s,u^s)$ satisfies (\ref{eq-saddle}).
    In what follows we consider the case that Algorithm 3.1 generates an infinite sequence $\{(x^s,u^s)\}$.
    Let $(\hat x,\hat u)$ be a cluster point of $\{(x^s,u^s)\}$. For the simplicity of notation, we assume that $(x^s,u^s)\rightarrow(\hat x,\hat u)$. 
    If $(\hat x,\hat u)$ is not a saddle point, then it violates one of the inequalities in (\ref{eq-saddle}).
    Without loss of generality, 
    consider the case that the first inequality of (\ref{eq-saddle}) is violated, that is, there exists $x_0$ such that 
    $$
    \hat u(x_0)+\bbe_P[\hat u(\xi-x_0)]>\hat u(\hat x)+\bbe_P[\hat u(\xi-\hat x)].
    $$
    Since $\hat u$ is continuous, then for sufficiently large $s$,
    $$
    u^s(x_0)+\bbe_P[u^s(\xi-x_0)]>u^s(x^s)+\bbe_P[u^s(\xi-x^s)],
    $$
    which is a contradiction to (\ref{eq-saddle1}).
    In the same manner, we can show that $(\hat x,\hat u)$ satisfies the second inequality in (\ref{eq-saddle}).
    The proof is complete.
    \hfill $\Box$
\end{proof}

Note that the cluster point is indeed a saddle point of the maximin problem (\ref{eq-RCE-zeta-N-discrete}) and existence of the latter is guaranteed by the fact that the objective function is linear in $u$ and concave in $x$.
Problem (\ref{eq-nonconcave-subproblem1}) is a convex problem because $\mathbb{B}_K(u^0_N,r)$ is a compact and convex set. By writing each utility function $u_N\in\mathscr U_N$ as
\begin{equation}\label{eq-nonconcave-utilityfunction}
u_N(t)=(a_1t+b_1)\mathds{1}_{[t_1,t_2]}(t)
	+\sum_{j=2}^{N-1}(a_jt+b_j)\mathds{1}_{(t_j,t_{j+1}]}(t)
\end{equation}
for $t\in[a,b]$ and writing down the Lagrange dual of
problem (\ref{eq:Kant-u-v-LP}),
\begin{subequations}
	\label{eq:Kant-u-v-LP-dual}
	\begin{eqnarray}
		\displaystyle \min_{\substack{\lambda^i_j, i=1,2,3,4 \\ j=2,\cdots,N}}
		&& -\frac{1}{2}\sum_{j=2}^N (\lambda_j^1+\lambda_j^2+\lambda_j^3+\lambda_j^4)(t_{j}-t_{j-1})^2 \\
		{\rm s.t.}   ~~~~
		&&  (\beta_{j-1} -\beta_{j-1}^0 )+ (\lambda_j^1-\lambda_j^2+\lambda_j^3-\lambda_j^4)=0,  j=2,\cdots,N, \\
		&& (\lambda_{j+1}^2-\lambda_{j+1}^1)(t_{j+1}-t_j) +(\lambda_j^4-\lambda_j^3)(t_j-t_{j-1})=0, \nonumber  \\
		&& \qquad \qquad \qquad \qquad \qquad \qquad \qquad \qquad \qquad j=2,\cdots,N-1, \\
		&& (\lambda_2^2-\lambda_2^1)(t_2-t_1)=0, \\
		&& (\lambda_N^4-\lambda_N^3)(t_N-t_{N-1})=0, \\
		&& \lambda_j^i\leq 0, j=2,\cdots,N, i=1,2,3,4.
	\end{eqnarray}
\end{subequations}
We can effectively reformulate
problem (\ref{eq-nonconcave-subproblem1}) as a linear program:
{\color{black}
	\begin{subequations}\label{eq-nonconcave-subproblem1-var}
		\begin{eqnarray}
		(a^s,b^s) \in \displaystyle \arg \min_{\substack{a_j,b_j, \\ j=1,\cdots,N-1}}
		&&(a_1x^{s-1} +b_1) \mathds{1}_{[t_1,t_{2}]} (x^{s-1})+\sum_{j=2}^{N-1} (a_jx^{s-1}+b_j) \mathds{1}_{(t_j,t_{j+1}]} (x^{s-1})\nonumber\\
		&&+\sum_{k=1}^K p_k \Bigl\{ (a_1 (\xi^k-x^{s-1}) +b_1) \mathds{1}_{[t_1,t_{2}]} (\xi^k-x^{s-1})\Bigr.\nonumber\\
		&&\Bigl. + \sum_{j=2}^{N-1} (a_j (\xi^k-x^{s-1}) +b_j) \mathds{1}_{(t_j,t_{j+1}]} (\xi^k-x^{s-1})  \Bigr\} \nonumber\\
		\mbox{s.t.}~~~~
		&&a_{j-1} t_{j}+b_{j-1}=a_{j} t_{j}+b_{j}, j=2,\cdots,N-1, \label{eq-nonconcave-subproblem1-a}\\
		&&a_1t_1+b_1=0, \label{eq-nonconcave-subproblem1-b} \\
		&&a_{N-1}t_N+b_{N-1}=1, \label{eq-nonconcave-subproblem1-c} \\
		&&a_{j+1}\leq a_j, j=1,\cdots,N-2, \label{eq-nonconcave-subproblem1-con}\\
		&&0\leq a_j\leq L, j=1,\cdots,N-1, \label{eq-nonconcave-subproblem1-d}\\
		&&-\frac{1}{2}\sum_{j=2}^N (\lambda_j^1+\lambda_j^2+\lambda_j^3+\lambda_j^4)(t_{j}-t_{j-1})^2\leq r,\qquad \qquad \label{eq-nonconcave-subproblem1-e}\\
		&&  (a_{j-1} -a_{j-1}^0 )+ (\lambda_j^1-\lambda_j^2+\lambda_j^3-\lambda_j^4)=0,  j=2,\cdots,N, \label{eq-nonconcave-subproblem1-f} \\
		&& (\lambda_{j+1}^2-\lambda_{j+1}^1)(t_{j+1}-t_j) +(\lambda_j^4-\lambda_j^3)(t_j-t_{j-1})=0,  \nonumber \\
		&&  \qquad \qquad \qquad \qquad \qquad \qquad \qquad \qquad j=2,\cdots,N-1, \label{eq-nonconcave-subproblem1-g} \\
		&& (\lambda_2^2-\lambda_2^1)(t_2-t_1)=0, \label{eq-nonconcave-subproblem1-h}\\
		&& (\lambda_N^4-\lambda_N^3)(t_N-t_{N-1})=0, \label{eq-nonconcave-subproblem1-i} \\
		&& \lambda_j^i\leq 0, j=2,\cdots,N, i=1,2,3,4, \nonumber
		\end{eqnarray}
	\end{subequations}
}
where $a_{j-1}^0$ denotes the slope of $u_N^0$ at interval $[t_{j-1},t_j]$.
Constraint (\ref{eq-nonconcave-subproblem1-a}) requires the piecewise linear function to be continuous at the kinks, constraint (\ref{eq-nonconcave-subproblem1-b}) and (\ref{eq-nonconcave-subproblem1-c}) represent the normalized conditions, (\ref{eq-nonconcave-subproblem1-con}) requires the concavity of utility function, (\ref{eq-nonconcave-subproblem1-d}) represents the Lipschitz condition, constraints (\ref{eq-nonconcave-subproblem1-e})-(\ref{eq-nonconcave-subproblem1-i}) represent the bounded Kantorovich ball.
Note that here we use the Lagrange dual problem  (\ref{eq:Kant-u-v-LP-dual})
instead of the primal problem  (\ref{eq:Kant-u-v-LP})
because the latter would have bilinear terms
$(\beta_{j-1}-\beta_{j-1}^0)y_{j-1}$ otherwise.

\section{RMOCE with non-piecewise linear utility functions}
\label{sec:error}

The computational schemes that we discussed in the previous section are applicable to the case when the ambiguity set is
constructed by a Kantorovich ball of piecewise linear utility functions. In practice, the utility functions are not necessarily
piecewise linear. This raises a question as to how much we may miss if we use ${\rm (RMOCE-PLU)}$
to compute (RMOCE) with the ambiguity set constructed by
the Kantorovich ball of general utility functions.
In this section, we address the issue which is essentially about
error bound of modelling error.
To maximize the scope of coverage, we consider $\zeta$-ball instead of the Kantovich ball.
Let $\mathscr U_L$ be a class of continuous, non-decreasing, concave functions
defined over $[a,b]$ with Lipschitz condition with moludus $L$ and normalized conditions $u(a)=0$ and $u(b)=1$. For $u^0\in\mathscr U_L$, we define
\begin{equation}
\mathbb{B}(u^0,r) :=\left\{u\in \mathscr{U}_L \;|\; \dd_\mathscr{G}(u,u^0)\leq r\right\},
\label{eq:zeta-ball}
\end{equation}
where 
\begin{equation}
\displaystyle{\dd_\mathscr G (u,v):= \sup_{g\in\mathscr G}|\langle g,u\rangle-\langle g,v\rangle|},
\label{pseu-metric}
\end{equation}
$\mathscr G$ is a set of measurable functions defined over $\R$ and
$\langle g,u\rangle:=\int_{\R} g(t)du(t)$.
$\dd_{\mathscr{G}}$ is known as a pseudo metric.
It can be observed that $\dd_\mathscr G (u,v)=0$ if and only if $\langle g,u\rangle=\langle g,v\rangle$ for all $g\in\mathscr G$ but not necessarily $u=v$ unless $\mathscr{G}$ is sufficiently large.  By specifying particular
properties of functions in set $\mathscr{G}$, we may obtain  some specific metric such as Kantorovich metric $\dd_K$ and the Kolmogorov metric
with
$\mathscr{G}= \displaystyle{\mathscr G_{I}}$, where
$\mathscr{G}_I$ consists of all indicator functions defined as
\begin{equation}
\mathds{1}_{(a, t]}(s):=
\begin{cases}
	1 & \text{if }s\in (a, t], \\
	0 & \text{otherwise}.
\end{cases}
\label{eq:idicator-fncts}	
\end{equation}

With the definition of the $\zeta$-ball and $u^0\in\mathscr U_L$, we may define
the corresponding RMOCE as
\begin{equation}
{\rm (RMOCE(\zeta))}  \quad \;\; R(\xi):=\displaystyle{\max_{x\in \R} \min_{u\in \mathbb{B}(u^0,r)}  }  \;\;  u(x)+\bbe_P[u(\xi-x)]
\label{eq:RCE-zeta}
\end{equation}
and the one when the utility functions are restricted to be piecewise linear:
\begin{equation}
{\rm (RMOCE(\zeta,N))}  \quad \;\;
R_N(\xi):=\displaystyle{\max_{x\in \R}
	\min_{u_N\in \mathbb{B}_N(u_N^0,r)}  }  \;\;  u_N(x)+\bbe_P[u_N(\xi-x)],
\label{eq:RCE-zeta-N}
\end{equation}
where  \begin{equation}
\mathbb{B}_N(u^0_N,r) :=\left\{u_N\in \mathscr{U}_N: \dd_\mathscr{G}(u_N,u^0_N)\leq r\right\}.
\label{eq:zeta-ball-N}
\end{equation}
We investigate the difference between
$\mathbb{B}_N(u_N^0,r)$ and $\mathbb{B}(u^0,r)$
and its propagation to the optimal values. Let $\mathcal U_1$ and $\mathcal U_2$ be two sets of utility function, $\dd_{\mathscr G}(u,\mathcal U_1):=\inf_{\tilde u\in\mathcal U_1}\dd_{\mathscr G}(u,\tilde u)$ between $u$ and $\mathcal U_1$, $\mathbb D_{\mathscr G}(\mathcal U_1,\mathcal U_2):=\sup_{u\in\mathcal U_1}\dd_{\mathscr G}(u,\mathcal U_2)$ be the deviation distance of $\mathcal U_1$ from $\mathcal U_2$, and $\mathbb H_{\mathscr G}(\mathcal U_1,\mathcal U_2):=\max\{\mathbb D_{\mathscr G}(\mathcal U_1,\mathcal U_2),\mathbb D_{\mathscr G}(\mathcal U_2,\mathcal U_1)\}$ be the Hausdorff distance between $\mathcal U_1$ and $\mathcal U_2$.



\subsection{Error bound on the ambiguity set}



We start by quantifying the difference between the ambiguity sets.
To this effect, we need a couple of technical results.

\begin{prop}(\cite[Proposition 4.1]{GuX21})
	\label{p-PC-appr}
	For each fixed $u\in \mathscr{U}_L$,
	let $u_N\in\mathscr{U}_N$ be such that 
	$u_N(t_i)=u(t_i)$ for $i=1,...N$
	and
	\begin{equation}
	u_N(t) := u(t_{i-1}) + \frac{u(t_{i})- u(t_{i-1})}{t_i-t_{i-1}} (t-t_{i-1}) \; {\rm for} \; t\in [t_{i-1},t_i], \; i=2,\cdots,N.
	\label{eq:uN-OCE}
	\end{equation}
	Then
	\begin{equation}
	\|u_N-u\|_\infty:=\sup_{t\in [a,b]} |u_N(t)-u(t)| \leq L\beta_N,
	\label{uN-appr-u}
	\end{equation}
	where
	$\beta_N:= \max_{i=2,\cdots,N} (t_i-t_{i-1})$.
	Moreover, in the case when $\mathscr G=\mathscr G_K$, it holds that
	\begin{equation}
	\label{eq:u-N-appr-u-zeta-metric}
	\displaystyle{\dd_{\mathscr G}(u,u_N) \leq 2\beta_N}.
	\end{equation}
	In the case when $\mathscr{G}=\mathscr{G}_I$,
	$ \dd_{\mathscr G}(u,u_N) \leq L\beta_N$.
\end{prop}



Here and later on, we call $u_N$ defined in (\ref{eq:u-N-appr-u-zeta-metric}) as a projection of $u$ on $\mathscr{U}_N$.
Next, we quantify the deviation distance and 
Hausdorff distance between $\zeta$-balls in the $\mathscr{U}_L$ and 
$\mathscr{U}_N$.
\begin{lem}
	\label{lem-dist-ball}
	Let $u_N\in\mathscr{U}_N$, $u\in\mathscr{U}_L$ and $\delta$, $r$ be any positive numbers. Then the following holds:
	\begin{itemize}
		\item[(i)] $\mathbb{D}_{\mathscr{G}}(\mathbb{B}_N(u_N,r+\delta),\mathbb{B}_N(u_N,r))
		\leq\delta$,
		$\mathbb{D}_{\mathscr{G}}(\mathbb{B}(u,r+\delta),\mathbb{B}(u,r))\leq\delta$,
		
		\item[(ii)] If $u_N$ is defined as in {\rm(\ref{eq:uN-OCE})} and $\mathscr{G}=\mathscr{G}_K\cup\mathscr{G}_I$, then $\mathbb{H}_{\mathscr{G}}(\mathbb{B}_N(u_N,r),\mathbb{B}(u,r)) \leq \max\{2,L\}\beta_N$ and $\mathbb{D}_{\mathscr{G}}(\mathbb{B}(u_N,r+\delta),\mathbb{B}(u_N,r))
		\leq\delta+2\max\{2,L\}\beta_N$.
	\end{itemize}
\end{lem}

\begin{proof}
	The proof is similar to that of \cite{WX20}, here we include a sketch for self-containedness.
	
	Part (i). We only prove the first inequality, as the second one can be proved analogously. Let $\tilde u_N\in\mathbb{B}_N(u_N,r+\delta) \setminus \mathbb{B}_N(u_N,r)$ and $u_N^\lambda:=\lambda \tilde u_N+(1-\lambda)u_N\in\mathscr{U}_N$, where $\lambda:=r/\dd_{\mathscr{G}}(u_N,\tilde u_N)\in(0,1)$. By the definition of $\dd_{\mathscr{G}}$, we have $\dd_{\mathscr{G}}(u_N^\lambda,u_N)=\sup_{g\in\mathscr{G}}\langle g,u_N^\lambda-u_N\rangle=\lambda\dd_{\mathscr{G}}(u_N,\tilde u_N)=r$, which implies $u_N^\lambda\in\mathbb{B}_N(u_N,r)$. Thus
	\begin{eqnarray*}
	\dd_{\mathscr{G}}(\tilde u_N,\mathbb{B}_N(u_N,r)) &\leq& \dd_{\mathscr{G}}(\tilde u_N,u_N^\lambda)=(1-\lambda)\dd_{\mathscr{G}}(\tilde u_N,u_N) \\
	&=& \dd_{\mathscr{G}}(\tilde u_N,u_N)-r\leq r+\delta-r=\delta.
	\end{eqnarray*}
	Since $\dd_{\mathscr{G}}(\hat u_N,\mathbb{B}_N(u_N,r))=0$ for $\hat u_N\in\mathbb{B}_N(u_N,r)$, we have $\dd_{\mathscr{G}}(\hat u_N,\mathbb{B}_N(u_N,r))\leq\delta$ for all $\hat u_N\in\mathbb{B}_N(u_N,r+\delta)$. By the definition of $\mathbb{D}_{\mathscr{G}}$, we have
	$$\mathbb{D}_{\mathscr{G}}(\mathbb{B}_N(u_N,r+\delta),\mathbb{B}_N(u_N,r))
	=\sup_{\hat u_N\in\mathbb{B}_N(u_N,r+\delta)}\dd_{\mathscr{G}}(\hat u_N,\mathbb{B}(u_N,r))$$
	and hence (i) holds.

	Part (ii). Let $\tilde u_N\in\mathbb{B}_N(u_N,r)$. Under \Cref{p-PC-appr},
	$$\dd_{\mathscr{G}}(\tilde u_N,u)\leq\dd_{\mathscr{G}}(\tilde u_N,u_N)+\dd_{\mathscr{G}}(u_N,u)\leq r+\max\{2,L\}\beta_N,$$
	which implies $\mathbb{B}_N(u_N,r)\subset\mathbb{B}(u,r+\max\{2,L\}\beta_N)$. By Part (i),
	$$\mathbb{D}_{\mathscr{G}}(\mathbb{B}_N(u_N,r),\mathbb{B}(u,r))
	\leq\mathbb{D}_{\mathscr{G}}(\mathbb{B}(u,r+2\beta_N),\mathbb{B}(u,r))\leq \max\{2,L\}\beta_N.$$
	Similarly, we have $\mathbb{D}_{\mathscr{G}}(\mathbb{B}(u,r),\mathbb{B}_N(u_N,r))\leq \max\{2,L\}\beta_N$. The result holds due to the definition of Hausdorff distance under $\zeta$-metric.
	
	Now we turn to prove $\mathbb{D}_{\mathscr{G}}(\mathbb{B}(u_N,r+\delta),\mathbb{B}(u_N,r))\leq\delta+2\max\{2,L\}\beta_N$.
	Since $u_N\in\mathscr{U}_N$, then we can find a $u\in\mathscr U_L$ such that $u_N$ is the projection of $u$. Hence for any $\tilde u\in\mathbb{B}(u_N,r+\delta)$, we have $\dd_{\mathscr{G}}(\tilde u,u)\leq\dd_{\mathscr{G}}(\tilde u,u_N)
	+\dd_{\mathscr{G}}(u_N,u)\leq r+\delta+\dd_{\mathscr{G}}(u_N,u)$. Consequently, $\mathbb{B}(u_N,r+\delta)\subset\mathbb{B}(u,r+\delta+\dd_{\mathscr{G}}(u_N,u))$.
	On the other hand, for any $\tilde u\in\mathbb{B}(u,r-\dd_{\mathscr{G}}(u_N,u))$, we have
	$$\dd_{\mathscr{G}}(\tilde u,u_N)\leq\dd_{\mathscr{G}}(\tilde u,u)+\dd_{\mathscr{G}}(u,u_N)
	\leq r-\dd_{\mathscr{G}}(u,u_N)+\dd_{\mathscr{G}}(u,u_N)=r,$$
	hence $\mathbb{B}(u,r-\dd_{\mathscr{G}}(u,u_N))\subset\mathbb{B}(u_N,r)$.
	Therefore, according to the definition of $\mathbb{D}_{\mathscr{G}}$ and Part (i),
	\begin{eqnarray*}	\mathbb{D}_{\mathscr{G}}(\mathbb{B}(u_N,r+\delta),\mathbb{B}(u_N,r))
	&\leq& \mathbb{D}_{\mathscr{G}}(\mathbb{B}(u,r+\delta+\dd_{\mathscr{G}}(u_N,u)),
	\mathbb{B}(u,r-\dd_{\mathscr{G}}(u,u_N))) \\
	&\leq& \delta+2\dd_{\mathscr{G}}(u,u_N)\leq\delta+2\max\{2,L\}\beta_N.
	\end{eqnarray*}
	The proof is complete.
	\hfill $\Box$
\end{proof}

With \Cref{lem-dist-ball}, we are ready to quantify the difference between $\mathbb{B}(u,r)$ and $\mathbb{B}_N(u_N,r)$.

\begin{theorem}
	\label{theorem-dist-ball}
	Let $u\in\mathscr U_L$ and $u_N$ is a projection of $u$ defined as in \rm(\ref{eq:uN-OCE}) and $\mathscr G=\mathscr G_K\cup\mathscr G_I$. Then
	\begin{equation}
	\mathbb{H}_{\mathscr{G}}(\mathbb{B}(u,r),\mathbb{B}_N(u_N,r))\leq5\max\{2,L\}\beta_N.
	\end{equation}
\end{theorem}

\begin{proof}
	By the triangle inequality of the Hausdorff distance in the space of $\mathscr U_L$, we have
	$$\mathbb{H}_{\mathscr{G}}(\mathbb{B}(u,r),\mathbb{B}_N(u_N,r))\leq
	\mathbb{H}_{\mathscr{G}}(\mathbb{B}(u,r),\mathbb{B}(u_N,r))
	+\mathbb{H}_{\mathscr{G}}(\mathbb{B}(u_N,r),\mathbb{B}_N(u_N,r)).$$
	From \Cref{lem-dist-ball}, $\mathbb{H}_{\mathscr{G}}(\mathbb{B}(u,r),\mathbb{B}(u_N,r))\leq \max\{2,L\}\beta_N$, so it suffices to show \linebreak $\mathbb{H}_{\mathscr{G}}(\mathbb{B}(u_N,r),\mathbb{B}_N(u_N,r))\leq 4\max\{2,L\}\beta_N$.
	By the definition of $\mathbb{D}_{\mathscr{G}}$,
	\begin{eqnarray*}
	\mathbb{D}_{\mathscr{G}}(\mathbb{B}(u_N,r),\mathbb{B}_N(u_N,r))
	&=& \sup_{\tilde u \in\mathbb{B}(u_N,r)}\dd_{\mathscr{G}}(\tilde u,\mathbb{B}_N(u_N,r)) \\
	&\leq& \sup_{\tilde u\in\mathbb{B}(u_N,r)}[\dd_{\mathscr{G}}(\tilde u,\tilde u_N)
	+\dd_{\mathscr{G}}(\tilde u_N,\mathbb{B}_N(u_N,r)] \\
	&\leq& \sup_{\tilde u\in\mathbb{B}(u_N,r)}[\max\{2,L\}\beta_N+\dd_{\mathscr{G}}(\tilde u_N,\mathbb{B}_N(u_N,r)] \\
	&\leq& \mathbb{D}_{\mathscr{G}}(\mathbb{B}_N(u_N,\max\{2,L\}\beta_N+r),\mathbb{B}_N(u_N,r))
	+\max\{2,L\}\beta_N \\
	&\leq& 2\max\{2,L\}\beta_N,
	\end{eqnarray*}
	where $\tilde u_N$ is the projection of $\tilde u$. The second inequality follows from (\ref{eq:u-N-appr-u-zeta-metric}), the third inequality is due to the fact that for any $\tilde u\in\mathbb{B}(u_N,r)$, its projection $\tilde u_N$ 
	satisfies 
	$$
	\dd_{\mathscr{G}}(\tilde u_N,u_N)\leq
	\dd_{\mathscr{G}}(\tilde u_N,\tilde u)+\dd_{\mathscr{G}}(\tilde u,u_N)
	\leq \max\{2,L\}\beta_N+r,
	$$
	that is, $\tilde u_N\in\mathbb{B}(u_N,\max\{2,L\}\beta_N+r)$. The last inequality follows from part (i) of \Cref{lem-dist-ball}. Likewise, we have
\begin{eqnarray*}
	\mathbb{D}_{\mathscr{G}}(\mathbb{B}_N(u_N,r),\mathbb{B}(u_N,r))
	&=& \sup_{\tilde u_N\in\mathbb{B}_N(u_N,r)}\dd_{\mathscr{G}}(\tilde u_N,\mathbb{B}(u_N,r)) \\
	&\leq& \sup_{\tilde u_N\in\mathbb{B}_N(u_N,r)}[\dd_{\mathscr{G}}(\tilde u_N,\tilde u)
	+\dd_{\mathscr{G}}(\tilde u,\mathbb{B}(u_N,r))] \\
	&\leq& \sup_{\tilde u_N\in\mathbb{B}_N(u_N,r)}\max\{2,L\}\beta_N+\dd_{\mathscr{G}}(\tilde u,\mathbb{B}(u_N,r)) \\
	&\leq& \sup_{\tilde u_N\in\mathbb{B}_N(u_N,r)}\max\{2,L\}\beta_N
	+\mathbb{D}_{\mathscr{G}}(\mathbb{B}(u_N,\max\{2,L\}\beta_N+r),\mathbb{B}(u_N,r)) \\
	&\leq& 4\max\{2,L\}\beta_N,
\end{eqnarray*}
	where the third inequality is derived from the fact that for any $\tilde u_N\in\mathbb{B}_N(u_N,r)$, that is, $\dd_{\mathscr{G}}(\tilde u_N,u_N)\leq r$, we have $\dd_{\mathscr{G}}(\tilde u,u_N) \leq \dd_{\mathscr{G}}(\tilde u,\tilde u_N)
	+\dd_{\mathscr{G}}(\tilde u_N,u_N) \leq \dd_{\mathscr{G}}(\tilde u,\tilde u_N)+r\leq \max\{2,L\}\beta_N+r$, that is $\tilde u\in\mathbb{B}(u_N,\max\{2,L\}\beta_N+r)$. The last inequality follows from part (ii) of \cref{lem-dist-ball}. Finally, by the definition of Hausdorff distance under metric $\dd_{\mathscr{G}}$, the proof is complete.
	\hfill $\Box$
\end{proof}

\subsection{Error bound on the optimal value}
\begin{theorem}
	Let $u^0\in\mathscr{U}_L$ and $u_N^0\in\mathscr{U}_N$ be defined as in {\rm(\ref{eq:uN-OCE})}. If $\mathscr G =\mathscr G_K\cup\mathscr G_I$ in (\ref{eq:zeta-ball-N}), then $|R(\xi)-R_N(\xi)|\leq10\max\{2,L\}\beta_N$.
\end{theorem}

\begin{proof}
	It is well known that
\begin{eqnarray*}
	|R_N(\xi)-R(\xi)|
	&\leq& \max_{x\in\mathds{R}}
	\left| \inf_{u\in\mathbb{B}(u^0,r)} \left\{u(x)+\bbe_P[u(\xi-x)]\right\}
	\right.\\
&&
\left. \quad \quad\quad-\inf_{u_N\in\mathbb{B}_N(u^0_N,r)}
	\left\{u_N(x)+\bbe_P[u_N(\xi-x)]\right\}\right|.
	\end{eqnarray*}
	Let $\delta$ be a small positive number. For any $x\in\mathds{R}$, we can find $u^{x}\in\mathbb{B}(u^0,r)$ and $u^{x}_N\in\mathbb{B}_N(u^0_N,r)$ depending on $\delta$ such that
\begin{eqnarray*}
	u^{x}(x)+\bbe_P[u^{x}(\xi-x)] &\leq& \inf_{u\in{\mathbb{B}(u^0,r)}} \{u(x)+\bbe_P[u(\xi-x)]\}+\delta, \\
	u^{x}_N(x)+\bbe_P[u^{x}_N(\xi-x)] &\geq& \inf_{u_N\in{\mathbb{B}_N(u_N^0,r)}}
	\{u_N(x)+\bbe_P[u_N(\xi-x)]\}, \\
	\sup_{t\in [a,b]}|u_N^{x}(t)-u^{x}(t)| &\leq& \mathbb{H}(\mathbb{B}_N(u_N^0,r),\mathbb{B}(u^0,r))+\delta,
\end{eqnarray*}
	where $\mathbb{H}$ denotes the Hausdorff distance in the space of continuous functions defined on $[a,b]$ equipped with infinity norm $\|\cdot\|_\infty$. Combining the above inequalities
\begin{eqnarray*}
	&& \inf_{u_N\in{\mathbb{B}_N(u^0_N,r)}}
	\{u_N(x)+\bbe_P[u_N(\xi-x)]\}-
	\inf_{u\in{\mathbb{B}(u^0,r)}}
	\{u(x)+\bbe_P[u(\xi-x)]\} \\
	&&\leq\bbe_P[u_N^{x}(\xi-x)-u^{x}(\xi-x)]+(u_N^{x}(x)-u^{x}(x))+\delta \\
	&&\leq 2\|u_N^{x}-u^{x}\|_\infty+\delta \\
	&&\leq 2\mathbb{H}(\mathbb{B}_N(u^0_N,r),\mathbb{B}(u^0,r))+3\delta.
\end{eqnarray*}
	By exchanging the positions of $\mathbb{B}_N(u^0_N,r)$ and $\mathbb{B}(u^0,r)$, we have
		\begin{eqnarray*}
	&&\inf_{u\in{\mathbb{B}(u^0,r)}}
	\{u(x)+\bbe_P[u(\xi-x)]\}
	-\inf_{u_N\in{\mathbb{B}_N(u^0_N,r)}}
	\{u_N(x)+\bbe_P[u_N(\xi-x)]\}\\
	&& \leq 2\mathbb{H}(\mathbb{B}(u^0,r),\mathbb{B}_N(u^0_N,r))+3\delta.
	\end{eqnarray*}
	Since $\delta\geq0$ can be arbitrarily small, we obtain
	\begin{eqnarray*}
	|R_N(\xi)-R(\xi)| &\leq& \max_{x\in\mathds{R}}\left|\inf_{u\in{\mathbb{B}(u^0,r)}}
	\{u(x)+\bbe_P[u(\xi-x)]\}
	-\inf_{u\in{\mathbb{B}_N(u_N^0,r)}}
	\{u(x)+\bbe_P[u(\xi-x)]\}\right| \\
	&\leq& 2\mathbb{H}(\mathbb{B}(u^0,r),\mathbb{B}_N(u^0_N,r)).
	\end{eqnarray*}
	The main challenge here is that $\mathbb{H}(\mathbb{B}(u^0,r),\mathbb{B}_N(u^0_N,r))$ differs from \linebreak $\mathbb{H}_{\mathscr{G}}(\mathbb{B}(u^0,r),\mathbb{B}_N(u^0_N,r))$. In what follows, we show that
	\begin{equation}\label{eq-Haus-I=Haus-inf}
	\mathbb{H}_{\mathscr{G}_I}(\mathbb{B}(u^0,r),\mathbb{B}_N(u^0_N,r))
	=\mathbb{H}(\mathbb{B}(u^0,r),\mathbb{B}_N(u^0_N,r)),
	\end{equation}
	where $\mathscr{G}_I=\{\mathds{1}_{(a, t]}(\cdot):t\in[a,b]\}$.
	Let $\tilde u\in\mathbb{B}(u^0,r)$ and $\tilde u_N\in\mathbb{B}_N(u^0_N,r)$,
	\begin{eqnarray}
	\label{eq:theorem-appro}
	\dd_{\mathscr{G}_I}(\tilde u,\tilde u_N)
	&=&
	\sup_{g \in \mathscr{G}_I} \left|\int_{a}^{b} g(t) d\tilde u(t) - \int_{a}^{b} g(t) d\tilde u_N(t)\right|\nonumber\\
	&=& \sup_{t \in [a,b]} \left|\int_{a}^{t} 1 d\tilde u(s) - \int_{a}^{t} 1 d\tilde u_N(s)\right|\nonumber\\
	&=&\sup_{t \in [a,b]} \left|\tilde u(t)-\tilde u_N(t)-\tilde u(a) +\tilde u_N(a)\right|\nonumber\\
	&=&\sup_{t \in [a,b]} |\tilde u(t)-\tilde u_N(t)|.
	\end{eqnarray}
	The last equality is due to the fact that $u(a) =u_N(a)=0$.
	By taking infimum w.r.t. $\tilde u_N\in \mathbb{B}_N(u^0_N,r)$ and taking superemum w.r.t. $\tilde u\in \mathbb{B}(u^0,r)$ on both sides of the equality above, we obtain
	$
	\mathbb{D}_{\mathscr{G}_I}(\mathbb{B}(u^0,r),\mathbb{B}_N(u_N^0,r)) =\mathbb{D}(\mathbb{B}(u^0,r),\mathbb{B}_N(u_N^0,r)).
	$
	Swapping the positions between $\mathbb{B}(u^0,r)$ and $\mathbb{B}_N(u_N^0,r)$, we obtain (\ref{eq-Haus-I=Haus-inf}). Combining with \Cref{theorem-dist-ball}, we obtain 
	the conclusion.
	\hfill $\Box$
\end{proof}

\section{Extensions}
\label{sec:unbounded}

In this section, we discuss potential extensions of
the MOCE and RMOCE models by considering utility functions with unbounded domain and multivariate utility functions.

\subsection{Utility function with unbounded domain}

In some important applications such as finance and economics, the underlying random variables which represent market demand, stock price and rate of return often have unbounded support. This raises a question as to whether our proposed model and computational schemes in the previous sections can be effectively applied to these situations.
Here we discuss this issue.

We start by defining a set of nonconstant increasing function defined over $\R$ denoted by $\mathscr U_\infty$.  We no longer restrict the domain of $u$ to a bounded interval $[a,b]$.
Let $u^0\in\mathscr U_\infty$, the $\zeta$-ball in $\mathscr U_\infty$ is defined as
\begin{equation*}
\mathbb B_\infty(u^0,r) := \{ u\in\mathscr U_\infty \;|\; \dd_{\mathscr G}(u,u^0)\leq r\},
\end{equation*}
where $\dd_{\mathscr G}$ is the pseudo metric defined in (\ref{pseu-metric}) and $\mathscr G$ is a set of measurable function throughout this section.
The robust modified optimized certainty equivalent model based on $\mathbb B_\infty(u^0,r)$ is defined as
\begin{equation}
{\rm (RMOCE)_\infty}  \quad \;\; R_\infty(\xi):=\displaystyle{\max_{x\in X} \min_{u\in \mathbb{B}_\infty(u^0,r)}  }  \;\;  u(x)+\bbe_P[u(\xi-x)],
\label{eq-RCE-unbounded}
\end{equation}
where $X$ is a compact implementable decisions over $X\subset\R$.
Our aim is to solve $\rm (RMOCE)_\infty$ and our concern is that the numerical schemes proposed in Section 3 cannot be applied to this problem directly.
Let $u^0_{\rm truc}$ be the truncation of $u^0$ over $[a,b]$, define
\begin{equation}
\mathbb B_{[a,b]}(u^0_{\rm truc},r) := \{u\in\mathscr U_{[a,b]} \;|\; \dd_{\mathscr G}(u,u^0_{\rm truc})\leq r\},
\end{equation}
where $\mathscr U_{[a,b]}$ denotes the set of nonconstant nondecreasing functions defined over $[a,b]$.
We rewrite (\ref{eq:RCE-zeta}) as
\begin{equation}
{\rm (RMOCE)_{[a,b]}}  \quad \;\; R_{[a,b]}(\xi):=\displaystyle{\max_{x\in X} \min_{u\in \mathbb{B}_{[a,b]}(u^0_{\rm truc},r)}  }  \;\;  u(x)+\bbe_P[u(\xi-x)].
\label{eq-RCE-bounded}
\end{equation}
What we are interested here is the difference between $\rm (RMOCE)_\infty$ and $\rm (RMOCE)_{[a,b]}$ in terms of the optimal value.
We will show that the difference between $R_\infty(\xi)$ and $R_{[a,b]}(\xi)$ is only related with the radius of the $\zeta$-ball under some moderate conditions and therefore we may solve $\rm (RMOCE)_\infty$ approximately by solving $\rm (RMOCE)_{[a,b]}$. The latter can be solved by the piecewise linear approximation scheme detailed in Section 3.

To build the bridge between $\mathbb{B}_{[a,b]}(u^0_{\rm truc},r)$ and $\mathbb{B}_\infty(u^0,r)$, we define the following set
\begin{equation}
\tilde {\mathbb B}_{[a,b]}(u^0,r) := \{u \in{\mathscr U}_\infty: \dd_{\mathscr G}(u,u^0)\leq r, u(t)=u(a) \text{~for~} t<a, \tilde u(t)=u(b) \text{~for~} t>b\}.
\end{equation}
Notice that $\tilde {\mathbb B}_{[a,b]}(u^0,r)$ is not a ball which is defined under the pseudo metric. Then we can establish the connection between $\tilde {\mathbb B}_{[a,b]}$ and $\mathbb B_\infty$ in the following proposition.

\begin{prop}\label{prop-ubounded-hausdorff}
	Let $u^0\in\mathscr U_\infty$ and assume that there exists a position number $\theta$ such that
	\begin{equation}
	\sup_{g\in\mathscr G, u\in\mathbb B_\infty(u^0,r)}\int_{\R}|g(t)|du(t)\leq\theta.
	\label{eq-condition-unbounded-hausorff}
	\end{equation}
	Then for any $\epsilon>0$ there exist constants $a<0$ and $b>0$ such that
	\begin{equation}
	\label{eq-conclusion-unbounded-hausdorff}
	\mathbb{H}_{\mathscr G} (\tilde {\mathbb B}_{[a,b]}(u^0,r+\epsilon),\mathbb B_\infty(u^0,r))\leq\sup_{g\in\mathscr G, u\in\mathbb B_\infty(u^0,r)}\left|\int_{\R\setminus[a,b]}g(t)du(t)\right|\leq\epsilon.
	\end{equation}
\end{prop}

\begin{proof}
	From the condition (\ref{eq-condition-unbounded-hausorff}), for any $\epsilon>0$ there exist constants $a<0$ and $b>0$ such that
	\begin{equation}
	\label{eq-unbounded-condition-1}
	\sup_{g\in\mathscr G, u\in\mathbb B_\infty(u^0,r)}\int_{\R\setminus[a,b]}|g(t)|du(t)\leq\epsilon.
	\end{equation}
	For any fixed $u\in\mathbb B_\infty(u^0,r)$ , let $\tilde u=u(t)$ for $t\in[a,b]$ and $\tilde u(t)=u(a)$ for $t<a$ and $\tilde u(t)=u(b)$ for $t>b$.
	Then we can obtain
	\begin{eqnarray}
	\dd_{\mathscr G}(\tilde u,u^0) &=&
	\sup_{g\in\mathscr G}
	\left|\int_{\R}g(t)d\tilde u(t)-\int_{\R}g(t)du^0(t)\right| \nonumber \\
	&=& \sup_{g\in\mathscr G}\left|
	\int_a^b g(t)d(\tilde u(t)-u^0(t))
	-\int_{\R\setminus[a,b]}g(t)du^0(t) \right| \nonumber\\
	&=& \sup_{g\in\mathscr G}\left|
	\int_a^b g(t)d(\tilde u(t)-u^0(t))
	+\int_{\R\setminus[a,b]}g(t)d(u(t)-u^0(t))-\int_{\R\setminus[a,b]}g(t)du(t)
	\right| \nonumber\\
	&\leq& r+\sup_{g\in\mathscr G}\left|\int_{\R\setminus[a,b]}g(t)du(t)\right| \leq r+\epsilon. \nonumber
	\end{eqnarray}
	Hence
	$\tilde u\in\tilde {\mathbb B}_{[a,b]}(u^0,r+\epsilon)$ and
	\begin{equation}
	\label{eq-proof-unbounded-hausdorff-distance}
	\mathbb D_{\mathscr G}(u,\tilde {\mathbb B}_{[a,b]}(u^0,r+\epsilon))
	\leq \dd_{\mathscr G}(u,\tilde u)
	= \sup_{g\in\mathscr G}\left|\int_{\R\setminus[a,b]}g(t)du(t)\right|
	\leq \epsilon.
	\end{equation}
	By taking supremum w.r.t. $u$ over $\mathbb B_\infty(u^0,r)$ on both sides of (\ref{eq-proof-unbounded-hausdorff-distance}), we obtain
	$$
	\mathbb D_{\mathscr G}(\mathbb B_\infty(u^0,r),\tilde {\mathbb B}_{[a,b]}(u^0,r+\epsilon)) \leq \epsilon.
	$$
	Note that $\tilde {\mathbb B}_{[a,b]}(u^0,r+\epsilon)\subset \mathbb B_\infty(u^0,r+\epsilon)$,
	then we have
	$$
	\mathbb D_{\mathscr G}(\tilde {\mathbb B}_{[a,b]}(u^0,r+\epsilon),\mathbb B_\infty(u^0,r)) \leq \mathbb D_{\mathscr G}(\mathbb B_\infty(u^0,r+\epsilon),\mathbb B_\infty(u^0,r))\leq\epsilon,
	$$
	where the last inequality is from part (i) of \cref{lem-dist-ball}. Consequently
	(\ref{eq-conclusion-unbounded-hausdorff}) follows.
	\hfill $\Box$
\end{proof}

From \Cref{prop-ubounded-hausdorff}, we can see that when the interval $[a,b]$ is large enough, the difference between $\tilde {\mathbb B}_{[a,b]}(u^0,r+\epsilon)$ and $\mathbb B_\infty(u^0,r)$ will not be significant.
We now turn to compare $\tilde {\mathbb B}_{[a,b]}(u^0,r+\epsilon)$ with $\mathbb B_{[a,b]}(u^0_{\rm truc},r)$, and
we could get similar conclusion in \cite[Section 6.2]{GuX21} that the extended function $\tilde u$ of $u\in\mathbb B_{[a,b]}(u^0_{\rm truc},r)$ is in $\tilde {\mathbb B}_{[a,b]}(u^0,r+\epsilon)$, where $\epsilon\geq\sup_{g\in\mathscr G}\left|\int_{\R\setminus[a,b]}g(t)du^0(t)\right|$, because $u^0\in\mathbb B_\infty(u^0,r)$.
By exploiting the relationship, we can quantify the difference between $R_\infty(\xi)$ and $R_{[a,b]}(\xi)$ in the following theorem.

\begin{theorem}
	Assume there exists a constant $\delta>0$ such that
	\begin{equation}
	\label{eq-condition-theorem-unbounded-optimalvalue-distance}
	\sup_{u\in\mathbb B_\infty(u^0,r+\delta),x\in X} \int_{\R}|u(\xi-x)|P(d\xi)<\infty
	\end{equation}
	and the condition in \Cref{prop-ubounded-hausdorff} is fulfilled.
	Then for any $\epsilon>0$, there exist constants $a<0$ and $b>0$ such that
	\begin{equation}
	|R_\infty(\xi)-R_{[a,b]}(\xi)|\leq 3\epsilon.
	\label{eq-result-theorem-unbounded-optimalvalue}
	\end{equation}
\end{theorem}

\begin{proof}
	It follows from conditions (\ref{eq-condition-theorem-unbounded-optimalvalue-distance}) and (\ref{eq-condition-unbounded-hausorff}) that for any $0<\epsilon<\delta$ there exist constants $a<0<b$ such that
	\begin{equation}
	\sup_{u\in\mathbb B_\infty(u^0,r+\delta), x\in X} \int_{\xi-x\in\R\setminus[a,b]}|u(\xi-x)|P(d\xi)\leq\epsilon/3
	\end{equation}
	and (\ref{eq-unbounded-condition-1}) holds.
	Since $\tilde{\mathbb B}_{[a,b]}(u^0,r+\epsilon)\subset\mathbb B_\infty(u^0,r+\delta)$, the above inequality implies
	\begin{equation}
	\sup_{u\in\tilde{\mathbb B}_{[a,b]}(u^0,r+\epsilon)}\bigl(|u(a)|P((-\infty,a))+|u(b)|P((b,+\infty))\bigr)\leq\epsilon/3.
	\end{equation}
	By definitions of $R_\infty(\xi)$ and $R_{[a,b]}(\xi)$,
		\begin{eqnarray*}
	 &&|R_\infty(\xi)-R_{[a,b]}(\xi)|\nonumber \\
	&\leq& \sup_{x\in X}\left|\inf_{u\in\mathbb B_\infty(u^0,r)} \left[u(x)+\int_{\R}u(\xi-x)P(d\xi)\right] -\inf_{\hat u\in\mathbb B_{[a,b]}(u^0_{\rm truc},r)} \left[\hat u(x)+\int_{\xi-x\in[a,b]}\hat u(\xi-x)P(d\xi)\right]\right| \nonumber\\
	&\leq&\sup_{x\in X}\left|\inf_{u\in\mathbb B_\infty(u^0,r)} \left[u(x)+\int_{\R}u(\xi-x)P(d\xi)\right] -\inf_{\tilde{u}\in\tilde{\mathbb B}_{[a,b]}(u^0,r+\epsilon)} \left[\tilde{u}(x)+\int_{\R}\tilde{u}(\xi-x)P(d\xi)\right]\right| \nonumber \\
	&+&\sup_{x\in X}\left|\inf_{\tilde{u}\in\tilde{\mathbb B}_{[a,b]}(u^0,r+\epsilon)}
	\left[\tilde{u}(x)+\int_{\R}\tilde{u}(\xi-x)P(d\xi)\right]
	-\inf_{\hat u\in\mathbb B_{[a,b]}(u^0_{\rm truc},r)}
	\left[\hat u(x)+\int_{\xi-x\in[a,b]}\hat u(\xi-x)P(d\xi)\right]\right|.
		\end{eqnarray*}
	Let us estimate the first term at the right side of the last inequality above.
	Observe that
	\begin{eqnarray*}
	&& \inf_{u\in\mathbb B_\infty(u^0,r)} \left[u(x)+\int_{\R}u(\xi-x)P(d\xi)\right]
	-\inf_{\tilde{u}\in\tilde{\mathbb B}_{[a,b]}(u^0,r+\epsilon)}
	\left[\tilde{u}(x)+\int_{\R}\tilde{u}(\xi-x)P(d\xi)\right] \nonumber\\
	&& \leq\inf_{u\in\mathbb B_\infty(u^0,r)}\sup_{\tilde{u}\in\tilde{\mathbb B}_{[a,b]}(u^0,r+\epsilon)}
	\left[\left|\int_{\R}u(\xi-x)P(d\xi)
	-\int_{\R}\tilde{u}(\xi-x)P(d\xi)\right|+|u(x)-\tilde{u}(x)|\right] \nonumber\\
	&& \leq\inf_{u\in\mathbb B_\infty(u^0,r)}\sup_{\tilde{u}\in\tilde{\mathbb B}_{[a,b]}(u^0,r+\epsilon)}
	\Bigg[\left|\int_{\xi-x\in[a,b]} u(\xi-x)P(d\xi)-\int_{\xi-x\in[a,b]}\tilde{u}(\xi-x)P(d\xi)\right|  \nonumber\\
	&& \qquad\qquad\qquad\qquad\qquad\qquad\qquad\qquad\qquad\qquad\qquad\qquad\qquad\qquad
	+|u(x)-\tilde{u}(x)|+\frac{2\epsilon}{3} \Bigg] \nonumber \\
	&& \leq\inf_{u\in\mathbb B_\infty(u^0,r)}\sup_{\tilde{u}\in\tilde{\mathbb B}_{[a,b]}(u^0,r+\epsilon)} \left[\sup_{t\in[a,b]}|u(t)-\tilde{u}(t)|+|u(x)-\tilde{u}(x)| 
	+\frac{2\epsilon}{3}\right] \nonumber\\
	&& \leq\inf_{u\in\mathbb B_\infty(u^0,r)}\sup_{\tilde{u}\in\tilde{\mathbb B}_{[a,b]}(u^0,r+\epsilon)}\left[\dd_{\mathscr{G}_I}(u,\tilde{u}) +|u(x)-\tilde{u}(x)|+\frac{2\epsilon}{3}\right].
	\end{eqnarray*}
	Thus
	\begin{eqnarray}
	&&\sup_{x\in X}\left|\inf_{u\in\mathbb B_\infty(u^0,r)}
	\left[u(x)+\int_{\R}u(\xi-x)P(d\xi)\right]
	-\inf_{\tilde{u}\in\tilde{\mathbb B}_{[a,b]}(u^0,r+\epsilon)}
	\left[\tilde{u}(x)+\int_{\R}\tilde{u}(\xi-x)P(d\xi)\right]\right| \nonumber \\
	 && \leq2\mathbb{H}_{\mathscr{G}_I}(\mathbb B_\infty(u^0,r),\tilde{\mathbb B}_{[a,b]}(u^0,r+\epsilon))+\frac{2\epsilon}{3}.
	 \label{eq-proof-unbounded-optimalvalue-distance}
	\end{eqnarray}
	The last inequality holds due to $X\subset[a,b]$.
	Now let us turn to the second term. For any $x\in X$ and a fixed positive number $\varepsilon$, we can find $\hat u_\xi\in\mathbb B_{[a,b]}(u^0_{\rm truc},r)$ and its extended function $\tilde u_\xi\in\tilde{\mathbb B}_{[a,b]}(u^0,r+\epsilon)$ such that
		\begin{eqnarray*}
	\hat u_\xi(x)+\int_{\xi-x\in[a,b]}\hat u_\xi(\xi-x)P(d\xi) &\leq& \inf_{\hat u\in\mathbb B_{[a,b]} (u^0_{\rm truc},r)}
	\left[\hat u(x)+\int_{\xi-x\in[a,b]}\hat u(\xi-x)P(d\xi)\right]+\varepsilon,\\ 
	\tilde u_\xi(x)+\int_{\R}\tilde u_\xi(\xi-x)P(d\xi) &\geq&
	\inf_{\tilde u\in\tilde{\mathbb B}_{[a,b]}(u^0,r+\epsilon)}
	\left[\tilde u(x)+\int_{\R}\tilde u(\xi-x)P(d\xi)\right].
		\end{eqnarray*}
	Consequently we have
	\begin{eqnarray}
	&&
	\inf_{\tilde{u}\in\tilde{\mathbb B}_{[a,b]}(u^0,r+\epsilon)}
	\left[\tilde{u}(x)+\int_{\R}\tilde{u}(\xi-x)P(d\xi)\right]
	-\inf_{\hat u\in\mathbb B_{[a,b]}(u^0_{\rm truc},r)}
	\left[\hat{u}(x)+\int_{\xi-x\in[a,b]}\hat u(\xi-x)P(d\xi)\right]
	\nonumber \\
	&&\leq \tilde u_\xi(x)+\int_{\R}\tilde u_\xi(\xi-x)P(d\xi)-\hat u_\xi(x)-\int_{\xi-x\in[a,b]}\hat u_\xi(\xi-x)P(d\xi)+\varepsilon \nonumber \\
	&&=\int_{\xi-x\in\R\setminus[a,b]}\tilde u_\xi(\xi-x)P(d\xi)+\varepsilon \nonumber \\
	&&\leq \sup_{\tilde u\in\tilde{\mathbb B}_{[a,b]}(u^0,r+\epsilon)}\Bigl(|\tilde u(a)|P((-\infty,a))+|\tilde u(b)|P((b,+\infty))\Bigr)+\varepsilon.
	\end{eqnarray}
	The second equality is satisfied because $\tilde u_{\xi}$ is the extended function of $\hat u_{\xi}$ and $x\in[a,b]$.
	By exchanging the positions of $\tilde{\mathbb B}_{[a,b]}(u^0,r+\epsilon)$ and $\mathbb B_{[a,b]}(u^0_{\rm truc},r)$, we have
 	 \begin{eqnarray}
 	&& \inf_{\hat u\in\mathbb B_{[a,b]}(u^0_{\rm truc},r)}
 	\left[\hat u(x)+\int_{\xi-x\in[a,b]}\hat u(\xi-x)P(d\xi)\right]-\inf_{\tilde{u}\in\tilde{\mathbb B}_{[a,b]}(u^0,r+\epsilon)}
 	\left[\tilde{u}(x)+\int_{\R}\tilde{u}(\xi-x)P(d\xi)\right] \nonumber \\
 	&& \leq \sup_{\tilde u\in\tilde{\mathbb B}_{[a,b]}(u^0,r+\epsilon)}\Bigl(|\tilde u(a)|P((-\infty,a))+|\tilde u(b)|P((b,+\infty))\Bigr)+\varepsilon.
 	 	\end{eqnarray}
	Since $\varepsilon$ can be arbitrarily small, we obtain
  \begin{eqnarray}
 &&    \sup_{x\in X}\Bigg|\inf_{\hat u\in\mathbb B_{[a,b]}(u^0_{\rm truc},r)}
	\left[\hat u(x)+\int_{\xi-x\in[a,b]}\hat u(\xi-x)P(d\xi)
	\right] \nonumber\\
&& \qquad\qquad\qquad\qquad\qquad\qquad\qquad\qquad -\inf_{\tilde{u}_2\in\tilde{\mathbb B}_{[a,b]}(u^0,r+\epsilon)}
	\left[\tilde{u}(x)+\int_{\R}\tilde{u}(\xi-x)P(d\xi)
	\right]\Bigg| \nonumber \\
	&&\leq\sup_{\tilde u\in\tilde{\mathbb B}_{[a,b]}(u^0,r+\epsilon)}(|\tilde u(a)|P((-\infty,a))+|\tilde u(b)|P((b,+\infty)))\leq\epsilon/3.
		\label{eq-proof-unbounded-optimalvalue-distance-thirdterm}
 \end{eqnarray}
	Combining (\ref{eq-proof-unbounded-optimalvalue-distance})
	-(\ref{eq-proof-unbounded-optimalvalue-distance-thirdterm}), we obtain (\ref{eq-result-theorem-unbounded-optimalvalue}) from
	(\ref{eq-conclusion-unbounded-hausdorff}).
	\hfill $\Box$
\end{proof}

{\color{black}
\subsection{Multiattribute utility case} 

The OCE models that we discussed so far are for single attribute decision making.
It might be interesting to ask whether the models can be extended to 
multi-attribute decision making. The answer is yes. Here we present two 
potential extended models. One is to consider  
the case that the utility function has an additive structure, that is, 
the multivariate utility function is the sum of the 
marginal utility functions of each attribute.
Such utility functions are widely used in the literature, 
see e.g.~\cite{KRM93,Abb09,AbS15}. In that case, 
{\color{black}
given $\xi:\Omega\rightarrow\R^m$ and $U:\R^m\rightarrow\R$,
we may define the MOCE as
\begin{equation}
{\rm (MMOCE-A)}\quad \quad   \displaystyle{M_u(\xi):=\sup_{x\in\R^m} \;\; U(x)+\bbe_P[U(\xi-x)]},
\label{eq:OCE-new-a}
\end{equation}
where the multiattribute utility function 
$U(x)=\sum_{i=1}^m u_i(x_i)$ and $u_i:\R\rightarrow\R$ is the marginal
utility function with respect to the $i$th attribute. The formulation can be simplified when the probability distribution of $\xi$ is the product 
of its marginal distributions:
\begin{equation}
\displaystyle{M_u(\xi)=  \sum_{i=1}^m \sup_{x_i\in\R} \;\;\{ u_i(x_i)+\bbe_{P_i}[u_i(\xi_i-x_i)]\}}.
\label{eq:OCE-new-a-1}
\end{equation}
The economic interpretation of the model is that the decision maker might have 
a portfolio of random assets $x_i$, $i=1,\cdots m$ and the 
DM would like to cash out $x_i$ from asset $i$. The marginal utilities may be the same or
different. Problem 
(\ref{eq:OCE-new-a-1})
is decomposable as it stands, thus it retains the properties outlined in \Cref{sec:Properties} and 
can be calculated by calculating 
$m$ single attribute MOCE simultaneously. 

When the utility function 
is non-additive, 
we may consider the following model:
\begin{equation}
{\rm (MMOCE-B)}\quad \quad   \displaystyle{M_u(\xi):=\sup_{t\in\R_+} \;\; \{u(td)+\bbe_P[u(\xi-td)]\}},
\label{eq:OCE-new-b}
\end{equation}
where $d$ is a fixed vector of weights. In this model, cash to be taken out from the assets
is in a prefixed proportion.  (MMOCE-B) is essentially a single variate 
MOCE model. 
Note that it is possible to further extend model (MMOCE-A) by replacing 
deterministic vector $x$ with a random vector $X$:
\begin{equation}
{\rm (MMOCE-A')}\quad \quad   \displaystyle{M_u(\xi):=\sup_{X} \;\; \bbe[U(X)]+\bbe[U(\xi-X)]}.
\label{eq:OCE-new-c}
\end{equation}
This kind of model has potential applications in finance where  
a firm detaches risk assets from non-risky assets in 
order to reduce the systemic risk
\cite{WXM21scenario}.
In that context, problem (MMOCE-A') is to find 
optimal separation $X$ from the existing overall portfolio of assets $\xi$. 
The problem is intrinsically two-stage, one may 
use linear/polynomial decision rule \cite{BaK11}
or K-adapativity method \cite{BeC07} 
 to obtain a (MMOCE-A)-version of approximation. 
Note also that model (MMOCE-A') is related to the IDR-based 
CDE model recently studied by Qi et al.~\cite{QCLP19} who use OCE  for optimizing individualised medical treatment. Since all of the  extended models outlined above require much more detailed analysis, we leave them for future research.

}

{\color{black}
	\section{Quantitative statistical robustness}
	\label{sec:quantitative}
	
	\subsection{Motivation}
	{\color{black} In \Cref{sec:numer-methods}, we discuss in detail how to
	obtain an approximate solution of (RMOCE) (to ease reading, we repeat the model here): 
	\begin{equation}
	\label{eq:RMOCE-P}
	{\rm (RMOCE-P)}\quad \quad \displaystyle{ R(P):=\max_{x\in X}\inf_{u\in\mathcal{U}} \bbe_P[u(x)+u(\xi-x)]},
	\end{equation}
	where $\xi$ follows probability distribution $P$.
	A key assumption is that the true probability distribution $P$ is known and discretely distributed.
	   This assumption may not be satisfied in  data-driven problems where
	   the true $P$ is unknown, and  one often uses empirical data to 
	   construct an approximation of $P$. Even worse is that such data may be contaminated.
	   
	   Let $\tilde \xi^1,...,\tilde \xi^N$ denote the contaminated empirical data (we call them {\em perceived data} and we use $N$ to denote the size of samples rather than number of breakpoints without causing confusion henceforth).
	   	   Let $Q_N:=\frac{1}{N}\sum_{i=1}^N\delta_{\tilde \xi_i}$ be 
	   	   the empirical distribution constructed with 
	   	   the perceived data, where $\delta_{\tilde {\xi}_i}$ is the Dirac measure at $\tilde {\xi}_i$.
	   We use the perceived data to solve the RMOCE model (assume that the model is solved precisely without computational error): 
	   	\begin{equation}
	\label{eq-empirical-RMOCE}
	{\rm (RMOCE-Q_N)}\quad \quad \displaystyle{ R(Q_N):=\max_{x\in X}\inf_{u\in\mathcal{U}} \bbe_{Q_N}[u(x)+u(\xi-x)]}.
	\end{equation}
	We then ask ourselves 
as to whether $R(Q_N)$  is a good estimation of $R(P)$ from statistical point of view.
This question is concerned with data perturbation rather than modelling/computational errors as discussed  in \Cref{sec:numer-methods}.

	To proceed the analysis, we introduce another empirical distribution, denoted by $P_N:=\frac{1}{N}\sum_{i=1}^N\delta_{\xi_i}$,
	which is constructed by the purified perceived data $\xi^1,...,\xi^N$ (the noise in the 
	perceived data is detached, we call them {\em real data} henceforth). 
	In practice, it is impossible to detach the noise, we introduce the notion purely for 
	the convenience of statistical analysis. 
	Let $R(P_N)$ be the optimal value of
	(RMOCE-P) by 
	replacing $P$ with $P_N$. 
	By the classical law of large numbers, we know that $P_N\to P$ and 
	$R(P_N)\to R(P)$ under moderate conditions. Thus in the literature of 
	stochastic programming, $R(P_N)$ is called a statistical estimator of $R(P)$
	and here we emphasize that this estimator is based on real data.
	
	Our question is then whether $R(Q_N)$ is close to $R(P_N)$ because the former is the only quantity that we are able to obtain. To address this question, we assume
	the perceived data are iid which means $Q_N\to Q$ for some $Q$ as $N\to \infty$.
	In other words, the perceived data may be viewed as if they are generated by the 
	invisible distribution $Q$. Let $R(Q)$ denote the optimal value of (RMOCE) with $P$ being replaced by $Q$. We then have 
	$$
    R(Q_N) -R(P_N) =  R(Q_N) -R(Q) + R(Q)-R(P)+ R(P)-R(P_N).
    $$
Thus if $ R(Q_N) \to R(Q)$ as $N\to\infty$ 
uniformly for all $Q$ close to $P$ 
and $R(Q)\to R(P)$ as $Q\to P$, then $R(Q_N)$ is close $R(P_N)$.
This explains roughly the motivation of this section. The formal quantitative statistical robust analysis is a bit more complex as we will examine the difference 
between the probability distributions of $R(Q_N)$ and $R(P_N)$ under some metric rather than 
estimating $ R(Q_N) -R(P_N)$ for each given set of perceived data.

}

\subsection{Statistical analysis}
	For any two probability measures $P, Q\in\mathscr{P}(\R)$, define the pseudo-metric between $P$ and $Q$ by
	\begin{equation}
	    \label{eq-zeta-metric}
	    \dd_{\mathscr{G}}(P,Q):=\sup_{g\in\mathscr{G}}\bigl|\bbe_{P}[g(\xi)]-\bbe_{Q}[g(\xi)]\bigr|.
	\end{equation}
	{\color{black} It can be seen that $\dd_{\mathscr{G}}(P,Q)$ is the maximal difference between the expected values of the class of measurable functions $\mathscr{G}$ with respect to $P$ and $Q$.}
	The specific pseudo metrics that we consider in this paper are the Fortet-Mourier metric and the Kantorovich metric.
	{\color{black}
	Recall that the $p$-th order Fortet-Mourier metric with $p\geq1$
	for $P,Q\in\mathscr P(\R)$:
	\begin{equation}
	\label{eq-fortet-moutier}
	\dd_{FM,p}(P,Q):=\sup_{g\in\mathscr G_p(\R)} \left|\int_{\R} g(\xi)P(d\xi)-\int_{\R} g(\xi)Q(d\xi)\right|,
	\end{equation}
	where 
	$$\mathscr G_p(\R) := \{g:\R\rightarrow\R\mid|g(\xi)-g(\tilde{\xi})|\leq c_p(\xi,\tilde{\xi})\|\xi-\tilde{\xi}\|, \forall \xi,\tilde{\xi}\in\R\}$$
	and 
	$$
	c_p(\xi,\tilde{\xi}):=\max\{1,\|\xi\|,\|\tilde{\xi}\|\}^{p-1},
	\quad \forall \xi,\tilde{\xi}\in\R.
	$$
	}
	{\color{black}
	When $p=1$, the functions in $\mathscr G_p(\R)$ are globally Lipschitz continuous with modulus $1$ and $\mathscr G_p(\R)$ coincides with $\mathscr G_K$ in (\ref{eq-kantorovich}). 
	Thus $\dd_{FM,1}(P,Q)=\dd_K(P,Q)$.
	For more details, see \cite{GiS02,Rom03,WXM21SR}.
	}
	
	
	To get the statistical robustness result, let $\R^{\otimes N}$ and $\mathcal B(\R)^{\otimes N}$ denote the Cartesian product $\R\times\cdot\cdot\cdot\times\R$ and its Borel sigma algebra. Let $P^{\otimes N}$ denote the probability measure on the measurable space $(\R^{\otimes N},\mathcal B(\R)^{\otimes N})$ with marginal $P$ on each $(\R,\mathcal B(\R))$ and $Q^{\otimes N}$ with marginal $Q$. 
	{\color{black}
	Now we can state the definition of statistical robustness of a statistic estimator, which 
	is proposed in \cite{GuX21,WXM21SR}.
	
\begin{definition}[Quantitative statistical robustness]
\label{def-statisticalrobustness}
Let
${\cal M} \subset \mathscr{P}(\R)$ be a set of probability measures.
 A sequence of statistical estimators $\hat{T}_N$ is said to be quantitatively statistically robust
on  $\mathcal{M}$ w.r.t. $(\dd_K,\dd_{FM,p})$
if 
there exists a positive constant $C$
such that for all $N$
\begin{eqnarray}
\label{eq:def-QSR}
\dd_K(P^{\otimes N}\circ \hat{T}_N^{-1},
Q^{\otimes N}\circ \hat{T}_N^{-1})\leq C\dd_{FM,p}(P,Q)<+\infty,
\; \forall P, Q \in \mathcal{M},
\end{eqnarray}
where $\dd_K$ is the Kantorovich metric on $\mathscr{P}(\R)$ and $\dd_{FM,p}$ is the Fortet-Mourier metric on $\mathscr{P}(\R)$.
\end{definition}

		Here 
		$P^{\otimes N}\circ\hat{T}_N^{-1}$
		and  $Q^{\otimes N}\circ\hat{T}_N^{-1}$
		are probability measures/distributions
		on $\R$.
	}
		{\color{black}
	The 
	next theorem states quantitative statistical robustness 
	of $\hat{R}_N :=R(Q_N)$.
	
	}
	\begin{theorem}
		Assume: (a) There exists a positive constant $L>0$ such that for all $x\in X$ and $u\in\mathcal U$,
		$$
		|u(\xi-x)-u(\xi'-x)|\leq L\max\left\{1,|\xi|,|\xi'|\right\}^{p-1}|\xi-\xi'|,
		$$
		(b)
		set $\mathcal U$ is chosen such that $\psi(t):=\sup_{u\in\mathcal U}|u(t)|$ is a gauge function, that is, $\psi:\R\rightarrow[0,\infty)$ is  continuous and $\psi\geq1$ holds outside a compact set.
		Then for any $N\in \Bbb{N}$,
		\begin{equation}
		\label{eq:conclusion-quantitative-statistical-robustness}
		\displaystyle{\dd_K(P^{\otimes N}\circ \hat{R}_N^{-1}, Q^{\otimes N}\circ \hat{R}_N^{-1})\leq L\dd_{FM,p}(P,Q), \forall P,Q\in\mathcal M^{\phi}},
		\end{equation}
		where $\phi(\xi):=C_0+L(|\xi|+|\xi|^p)$ for some constant $C_0>1$.
	\end{theorem}
	
	\begin{proof}
		By definition
		\begin{eqnarray}
		&&\dd_K(P^{\otimes N}\circ\hat{R}_N^{-1}, Q^{\otimes N}\circ\hat{R}_N^{-1}) \nonumber\\
		&& =\sup_{g\in\mathscr{G}_K}\left|\int_{\R}g(t)P^{\otimes N}\circ\hat{R}_N^{-1}(dt)-\int_{\R}g(t)Q^{\otimes N}\circ\hat{R}_N^{-1}(dt)\right| \nonumber\\
		&& =\sup_{g\in\mathscr{G}_K}\left|\int_{\R^{\otimes N}}g(\hat{R}(\boldsymbol{\xi}^N))P^{\otimes N}(d\boldsymbol{\xi}^N)-\int_{\R^{\otimes N}}g(\hat{R}(\boldsymbol{\xi}^N))Q^{\otimes N}(d\boldsymbol{\xi}^N)\right|,
		\label{eq:Kantorovich-distance}
		\end{eqnarray}
		where $\boldsymbol{\xi}^N=(\xi^1,...,\xi^N)$ and we write $\hat R(\boldsymbol{\xi}^N)$ for $\hat R_N$.
		To see the well-definedness of the pseudo-metric, notice that for every $g\in\mathscr G_K$ and a fixed $\boldsymbol{\xi}_0^N\in\R^{\otimes N}$
		\begin{equation}
		\label{eq:lipschitz-g}
		|g(\hat{R}(\boldsymbol{\xi}^N))|\leq|g(\hat{R}(\boldsymbol{\xi}_0^N))|+|\hat{R}(\boldsymbol{\xi}^N)-\hat{R}(\boldsymbol{\xi}_0^N)|,
		\end{equation}
		where $\boldsymbol{\xi}_0^N\in\R^{\otimes N}$ is fixed.
		From condition (b) and nondecreasing property of $u$, there exists a positive number $C_0$ such that
		\begin{equation}
		\label{eq:upperbound-objecticve}
		\sup_{u\in\mathcal U, x\in X}|u(x)+u(\xi-x)|\leq C_0+L(|\xi|+|\xi|^p), \forall \xi\in\R.
		\end{equation}
		By the definition of $\hat{R}(\boldsymbol{\xi}^N)$, it follows that
		$$
		|\hat{R}(\boldsymbol{\xi}^N)|=\left|\max_{x\in X} \inf_{u\in\mathcal U} \frac{1}{N}\sum_{k=1}^N[u(x)+u(\xi^k-x)]\right|\leq\frac{1}{N}\sum_{k=1}^N\phi(\xi^k).
		$$
		Moreover,
		\begin{eqnarray}
		\label{eq:boundness-vN}
		\int_{\R^{\otimes N}} |\hat{R}(\boldsymbol{\xi}^N)|P^{\otimes N}(d\boldsymbol{\xi}^N) &\leq & \int_{\R^{\otimes N}}\frac{1}{N}\sum_{k=1}^N\phi(\xi^k)P^{\otimes N}(d\boldsymbol{\xi}^N) \nonumber \\
		&=& \int_{\R}\phi(\xi)P(d\xi)<\infty, \forall P\in\mathcal M^{\phi},
		\end{eqnarray}
		where the equality holds due to the fact that $\xi^1,...,\xi^N$ are independent and identically distributed.
		Combining (\ref{eq:lipschitz-g}) and (\ref{eq:boundness-vN}) we can obtain
		$$
		\int_{\R^{\otimes N}}g(\hat{R}(\boldsymbol{\xi}^N))P^{\otimes N}(d\boldsymbol{\xi}^N)<\infty, \forall P\in\mathcal M^{\phi}.
		$$
		Similar argument can be made on $\int_{\R^{\otimes N}}g(\hat{R}(\boldsymbol{\xi}^N))Q^{\otimes N}(d\boldsymbol{\xi}^N)$ for any $Q\in\mathcal M^{\phi}$.
		Next,
		for any $P, Q\in\mathcal M^{\phi}$,
		\begin{eqnarray*}		
			|R(P)-R(Q)| &=& \left|\sup_{x\in X}\inf_{u\in\mathcal U} \{u(x)+\bbe_P[u(\xi-x)]\}
		-\sup_{x\in X}\inf_{u\in\mathcal U } \{u(x)+\bbe_Q[u(\xi-x)]\}\right| \\
		&\leq&  \sup_{x\in X}\sup_{u\in\mathcal U} |\bbe_P[u(\xi-x)]-\bbe_Q[u(\xi-x)]| \\
		&\leq& \dd_{FM,p} (P,Q),
		\end{eqnarray*}
		where the last inequality follows from condition (a).
		Then we can obtain
		\begin{multline}
			|g(\hat{R}(\tilde{\xi}^1,...,\tilde{\xi}^N)-g(\hat{R}(\hat{\xi}^1,...,\hat{\xi}^N))|
			\leq
			|\hat{R}(\tilde{\xi}^1,...,\tilde{\xi}^N)-\tilde{R}(\hat{\xi}^1,...,\hat{\xi}^N)| \\
			\leq\frac{1}{N}\sum_{k=1}^N \sup_{x\in X, u\in\mathcal U} |u(\tilde\xi^k-x)-u(\hat\xi^k-x)|
			\leq
			\frac{L}{N}
			\sum_{k=1}^N\max\{1+|\tilde{\xi}^k|+|\hat{\xi}^k|\}^{p-1}|\tilde\xi^k-\hat\xi^k|.
			\nonumber
		\end{multline}
		It follows by \cite[Lemma 4.4]{WXM21SR} that 
		\begin{eqnarray}
		(\ref{eq:Kantorovich-distance}) \leq \dd_{FM,p}(P^{\otimes N},Q^{\otimes N}) \leq L\dd_{FM,p}(P,Q)
		\end{eqnarray}
		and hence
		inequality (\ref{eq:conclusion-quantitative-statistical-robustness}).
		\hfill $\Box$
	\end{proof}
	
}

\section{Numerical tests}
\label{sec:numerical}
We have carried out some
tests on the numerical schemes for computing RMOCE.
In this section, we report the preliminary numerical results.

The first set of tests are about the comparison between the MOCE model (\ref{eq:OCE-new}) and  OCE model (\ref{eq:OCE}) in
terms of the optimal values and the optimal solutions.
We do so by considering $\xi$ following some specific distributions including 
uniform, Gamma, lognormal and normalized Pareto distribution.
The second set of tests are on the RMOCE model and numerical schemes proposed in \Cref{sec:numer-methods}.
We investigate how
the optimal value and the worst case utility function
in the RMOCE model
change as the radius of ambiguity set and the number of breakpoints vary.
We use the parallel particle swarm optimization method \cite{KeE95,MMC11} to solve problem (\ref{eq-RCE-zeta-N-discrete}) and CVX solver to solve inner minimization problem (\ref{eq-nonconcave-subproblem1}).
All the tests are carried out in Matlab R2021a installed on a PC (16GB RAM, CPU 2.3 GHz) with Intel Core i7 processor.

Throughout the section we restrict $\mathscr U$ to a set of all increasing concave utility functions mapping from a compact interval $[a,b]$. We take $[a,b]$ as
the domain of $u$ which is the union of ranges of $x$ and $\xi-x$ for $x\in[\xi_{\min}/2,\xi_{\max}/2]$ by  \Cref{prop-piecewiselinear-optimality} because the number $N$ of breakpoints can guarantee that $\beta_N\leq\xi_{\max}-\xi_{\min}$.
We generate iid samples
$\xi^1,...,\xi^K$ for random variable $\xi$
with equal probabilities $p_k = 1/K$ for $k=1,...,K$.

In the first set of tests of OCE and MOCE,
we set the nomial utility function as $u^0(t)=(1-e^{-2 t})/2$.
\Cref{table-MOCE}
displays the optimal values and the
optimal solutions as well as  the CPU times.
The $3$th and $4$th columns present the optimal values of OCE and MOCE model, and the $5$th and $6$th columns present the optimal solutions of OCE and MOCE model, respectively.
As we can see, the OCE values are consistently larger than the MOCE values, this is because $u(t)<t$. Moreover, we find
the optimal solutions of MOCE problem (under $x^*$) fall within
$[\xi_{\min}/2,\xi_{\max}/2]$ although we have not displayed
the intervals due to the limitation of space. This complies with
\Cref{Prop:Attain-max-RMOCE}.

\begin{table}[!hbp]
\footnotesize
	\setlength{\abovecaptionskip}{8pt}
	\setlength{\belowcaptionskip}{4pt}
	\centering
	\captionsetup{font={scriptsize}}
	\caption{Numerical results of MOCE}
	\renewcommand\arraystretch{1} 
	\begin{tabular}{ccccccc}
		\hline
		Distribution & K & $M_u(\xi)$ &  $S_u(\xi)$ & $x^*$ & $\eta^*$ & CPU time \\
		\hline
		\multirow{3}*{Uniform (-1,1)} & 10 & -0.5590 & -0.4440 & -0.2220 & -0.4441 & 0.8700   \\
		& 100 &  -0.1950 & -0.1782 & -0.0891 & -0.1782 & 0.9914 \\
		& 1000 & -0.3508 & -0.3008 & -0.1504 & -0.3008 & 3.8415 \\
		\hline
		\multirow{3}*{Lognormal (0,1)} & 10 & 0.4929 & 0.6792 & 0.3396 &  0.6792 & 0.4279  \\
		& 100 &  0.5182 & 0.7303 & 0.3651 & 0.7303 & 0.8139  \\
		& 1000 & 0.5313 & 0.7578 & 0.3789 & 0.7578  & 3.7001 \\
		\hline
		\multirow{3}*{Pareto (1,1.5)} & 10 & 0.8692 & 2.0337 & 1.0169 & 2.0337 & 0.4484  \\
		& 100 &  0.8990 & 2.2926 & 1.1463 & 2.2925  & 0.7263 \\
		& 1000 & 0.8942 & 2.2461 & 1.1231 & 2.2461 & 3.6693 \\
		\hline
		\multirow{3}*{Gamma (0.53,3)} & 10 & 0.3392 & 0.4143 & 0.2072 &  0.4143  &  0.5002 \\
		& 100 &  0.4415 & 0.5824 & 0.2912 &0.5825 &  0.7094  \\
		& 1000 & 0.4088 & 0.5255 & 0.2628 & 0.5255 & 3.7729   \\
		\hline
	\end{tabular}
	\label{table-MOCE}
\end{table}

In the second set of tests about RMOCE,
we set the nominal utility as $u^0(t)=(1-e^{-\alpha})/2$ where 
$\alpha\in\R_+$ is a parameter which 
determines the degree of concavity of the utility function.
The number of random samples is fixed at $K=100$ for the uniform distribution and $K=10$ for 
Gamma, lognormal and normalized Pareto distribution.
The parameters of the tests are listed in \Cref{table-parameters}, the 4th column represents the Lipschitz modulus of utility functions. 
For the cases where the random samples are generated by uniform distribution,
\Cref{fig-con-radius-utilityvalue,fig-con-radius-optimalvalue} visualize the worst case utility functions and the optimal values as the radius decreases.
\Cref{fig-con-N} visualizes the change of optimal values as the number of breakpoints increases.
It can be seen that the number of breakpoints has little effect on the optimal value.
For the cases when $\xi$ follows Gamma, lognormal and normalized Pareto distribution,
\Cref{fig-utilityfunction-distributions,fig-optimalvalue-distributions}
visualize the changes of the worst case utility functions and the optimal values as the radius decreases.
We can see that the worst utility function moves closer to the nominal utility function as the radius of the ambiguity
set decreases to zero, the optimal value increases as the radius decreases.
This is because the Kantorovich ball becomes smaller when the radius decreases. In the case that $r=0$, the worst case utility function is the piecewise linear approximation of the nominal utility function.
The error bound of the optimal value is also
depicted in \Cref{fig-con-N,fig-optimalvalue-distributions}, note that the error bound is getting smaller when the number of breakpoints increases in \Cref{fig-con-N}.
\Cref{table-N-runningtime} provides the optimal values and running time for different number of breakpoints.

\begin{table}[!hbp]
\setlength{\abovecaptionskip}{8pt}
\setlength{\belowcaptionskip}{4pt}
\footnotesize
\begin{minipage}{0.45\linewidth}
\centering
\captionsetup{font={scriptsize}}
\caption{Parameters of RMOCE tests}
\begin{tabular}{cccc}
    		\hline
		Distribution & $\alpha$  & N & L\\
		\hline
		Uniform (-1,1)  & 2  & 10 &  30  \\
		Lognormal (0,1)  & 1/2   & 300 & 10  \\
		Pareto (1,1.5) & 1/3  & 300 & 10  \\
		Gamma (0.53,3) & 1/2  & 300 & 10 \\
		\hline
\end{tabular} 
\label{table-parameters}
\end{minipage}
\qquad
\begin{minipage}{0.38\linewidth} 
\centering
\captionsetup{font={scriptsize}}
\caption{Running time for different number of breakpoints}
\begin{tabular}{ccc} 
		\hline
		N & Optimal value & CPU time \\
		\hline
		20 & -108.4846 & 20.7796  \\
		40 & -108.6524  & 24.1097 \\
		60 & -108.5553 & 31.4506 \\
		80 & -108.5657 & 36.7656  \\
		100 & -108.5648 & 41.9548  \\
		\hline
\end{tabular} 
\label{table-N-runningtime}
\end{minipage}
\end{table}

\begin{figure}[!hbp]
\setlength{\abovecaptionskip}{6pt}
\setcaptionwidth{1.5in}
  \begin{minipage}[t]{0.32\textwidth}
    \centering
    \includegraphics[width=5.5cm]{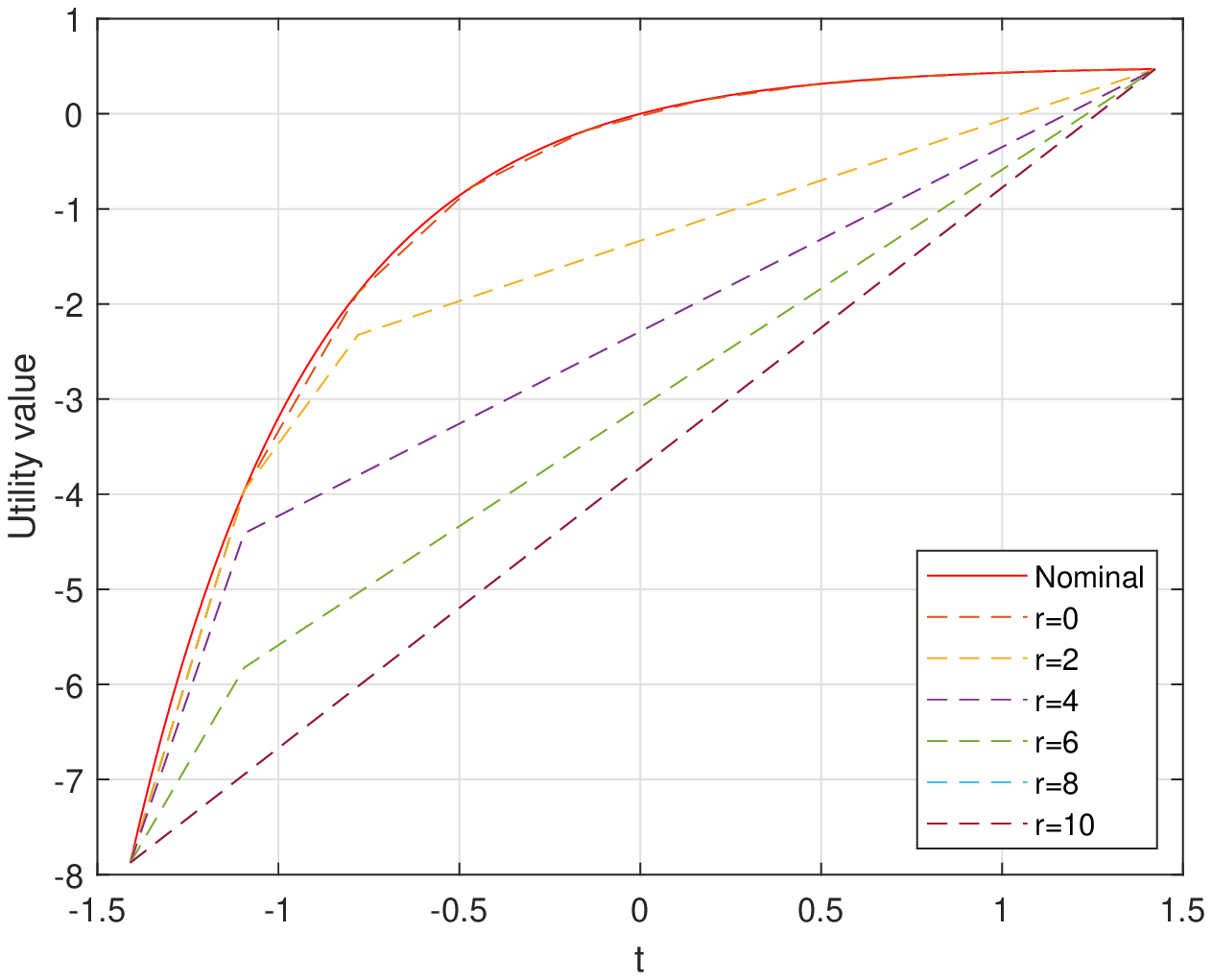}
    \captionsetup{font={scriptsize}}
    \caption{Worst case utility functions with different r}
    \label{fig-con-radius-utilityvalue}
  \end{minipage}%
  \begin{minipage}[t]{0.32\textwidth}
    \centering
    \includegraphics[width=5.5cm]{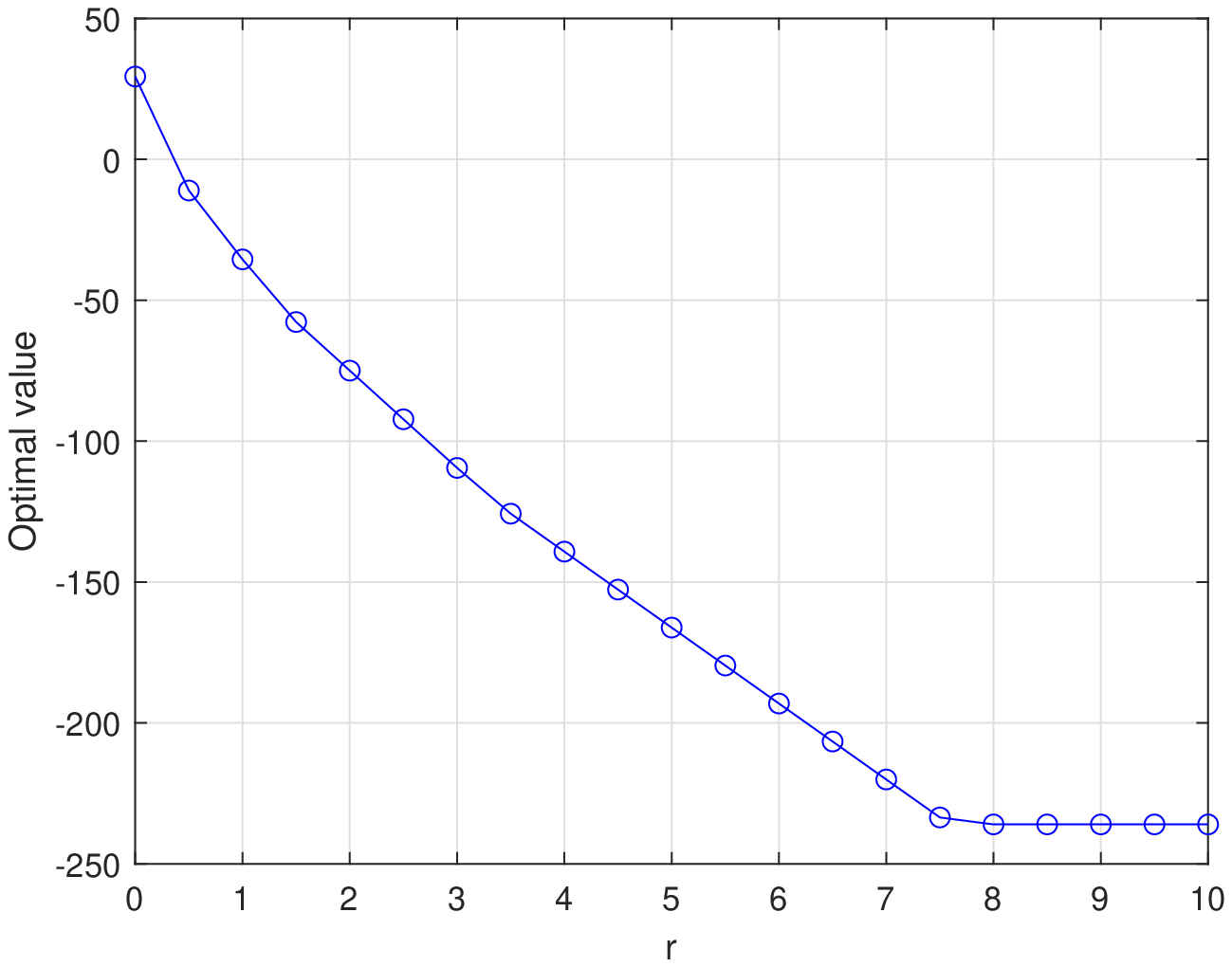}
    \captionsetup{font={scriptsize}}
    \caption{Optimal values with different r}
    \label{fig-con-radius-optimalvalue}
  \end{minipage}
  \begin{minipage}[t]{0.32\textwidth}
    \centering
    \includegraphics[width=5.5cm]{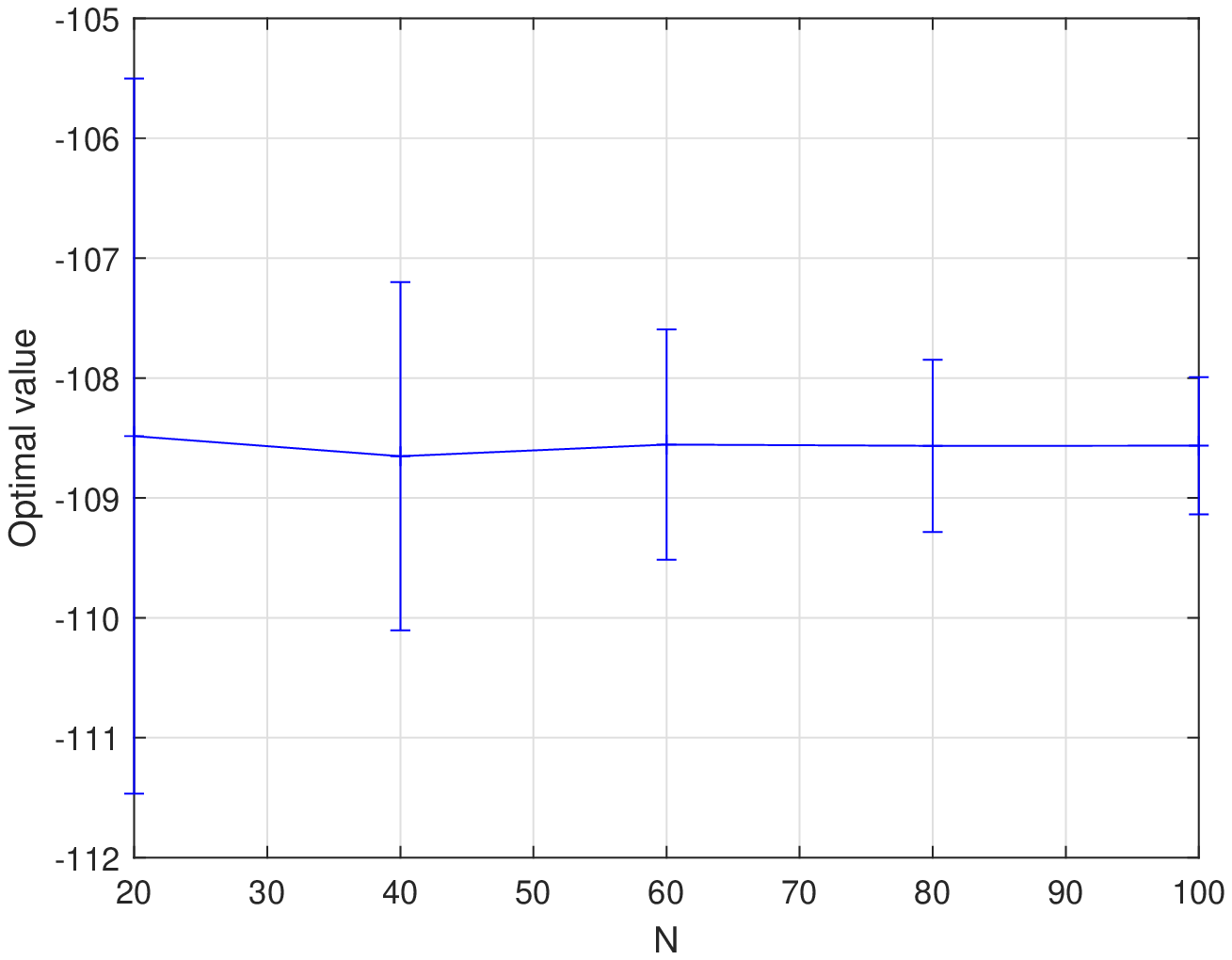}
    \captionsetup{font={scriptsize}}
    \caption{Optimal values with different N}
    \label{fig-con-N}
  \end{minipage}%
\end{figure}


\begin{figure}[!hbp]
	\centering
	\setlength{\abovecaptionskip}{4pt}
	\setlength{\belowcaptionskip}{4pt}
	\begin{minipage}[t]{0.32\textwidth}
		\centering
		\caption*{\scriptsize{lognormal (0,1)}}
		\includegraphics[width=5.5cm]{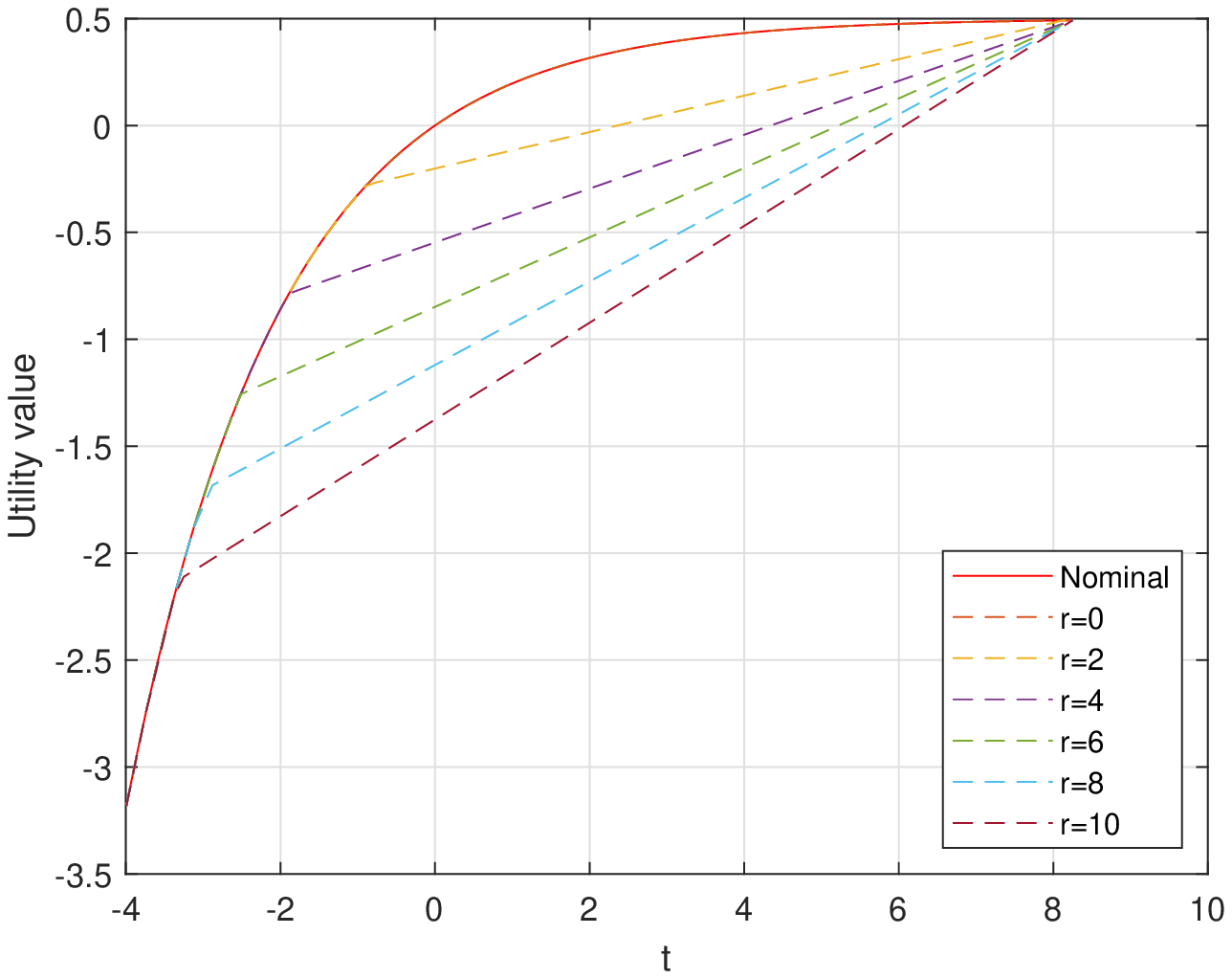}
	\end{minipage}
	\begin{minipage}[t]{0.32\textwidth}
		\centering
		\caption*{\scriptsize{Pareto (1,1.5)}}
		\includegraphics[width=5.5cm]{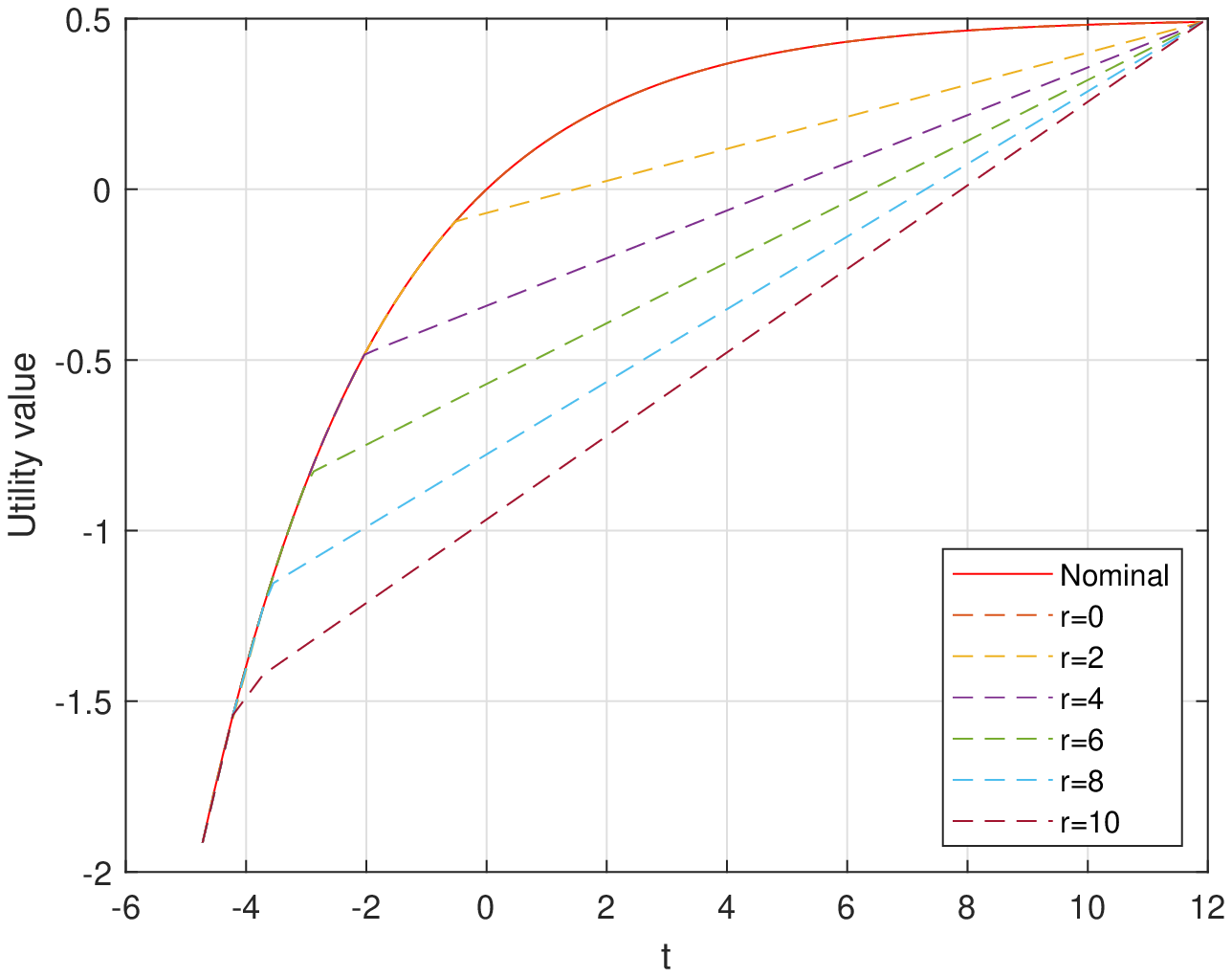}
	\end{minipage}
	\begin{minipage}[t]{0.32\textwidth}
		\centering
		\caption*{\scriptsize{Gamma (0.53,3)}}
		\includegraphics[width=5.5cm]{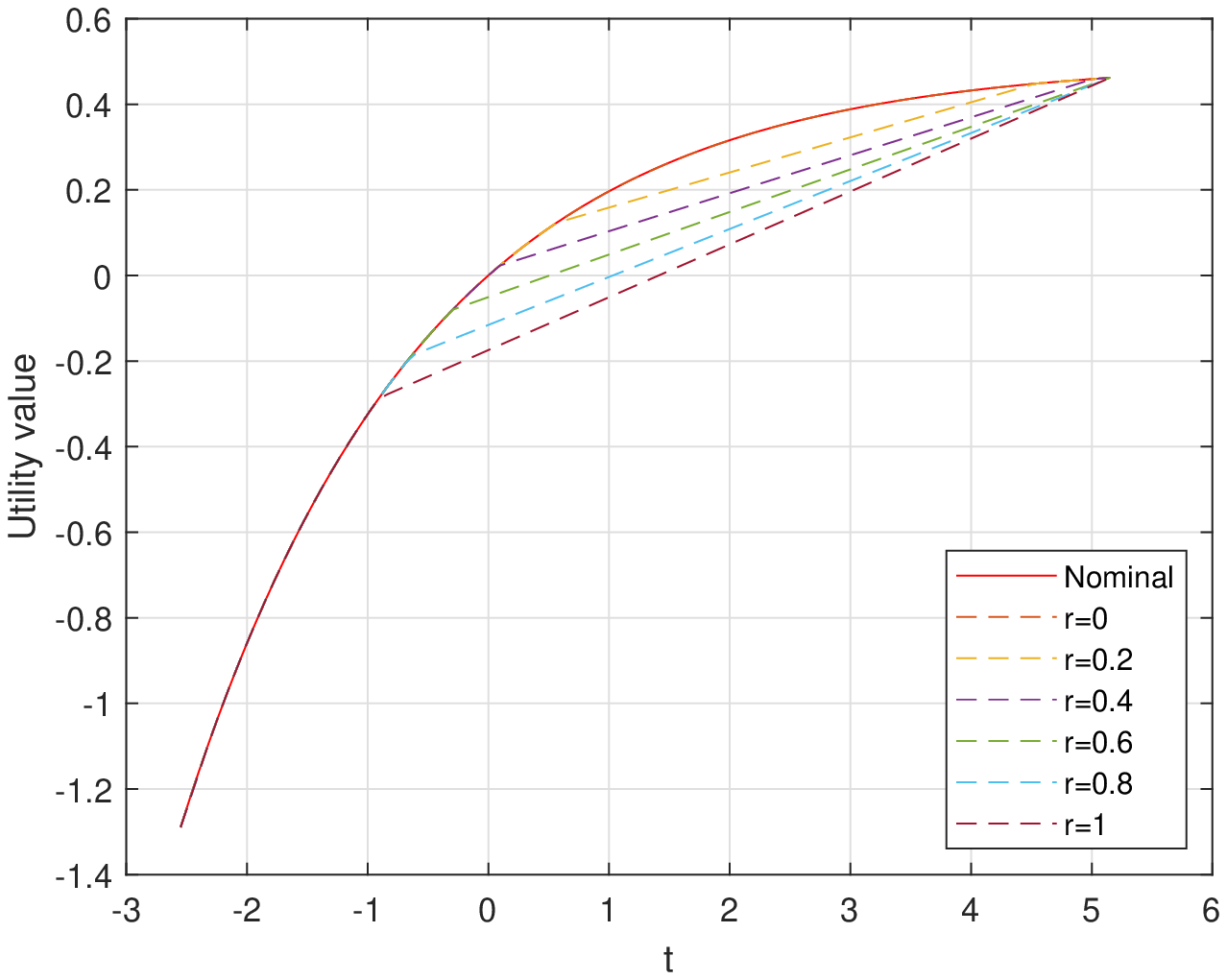}
	\end{minipage}
	\captionsetup{font={scriptsize}}
	\caption{Worst case utility functions with heavy-tailed distributions}
	\label{fig-utilityfunction-distributions}
\end{figure}

\newpage

\begin{figure}[!hbp]
	\centering
	\setlength{\abovecaptionskip}{4pt}
	\setlength{\belowcaptionskip}{4pt}
	\begin{minipage}[t]{0.32\textwidth}
		\centering
		\caption*{\scriptsize{lognormal (0,1)}}
		\includegraphics[width=5.5cm]{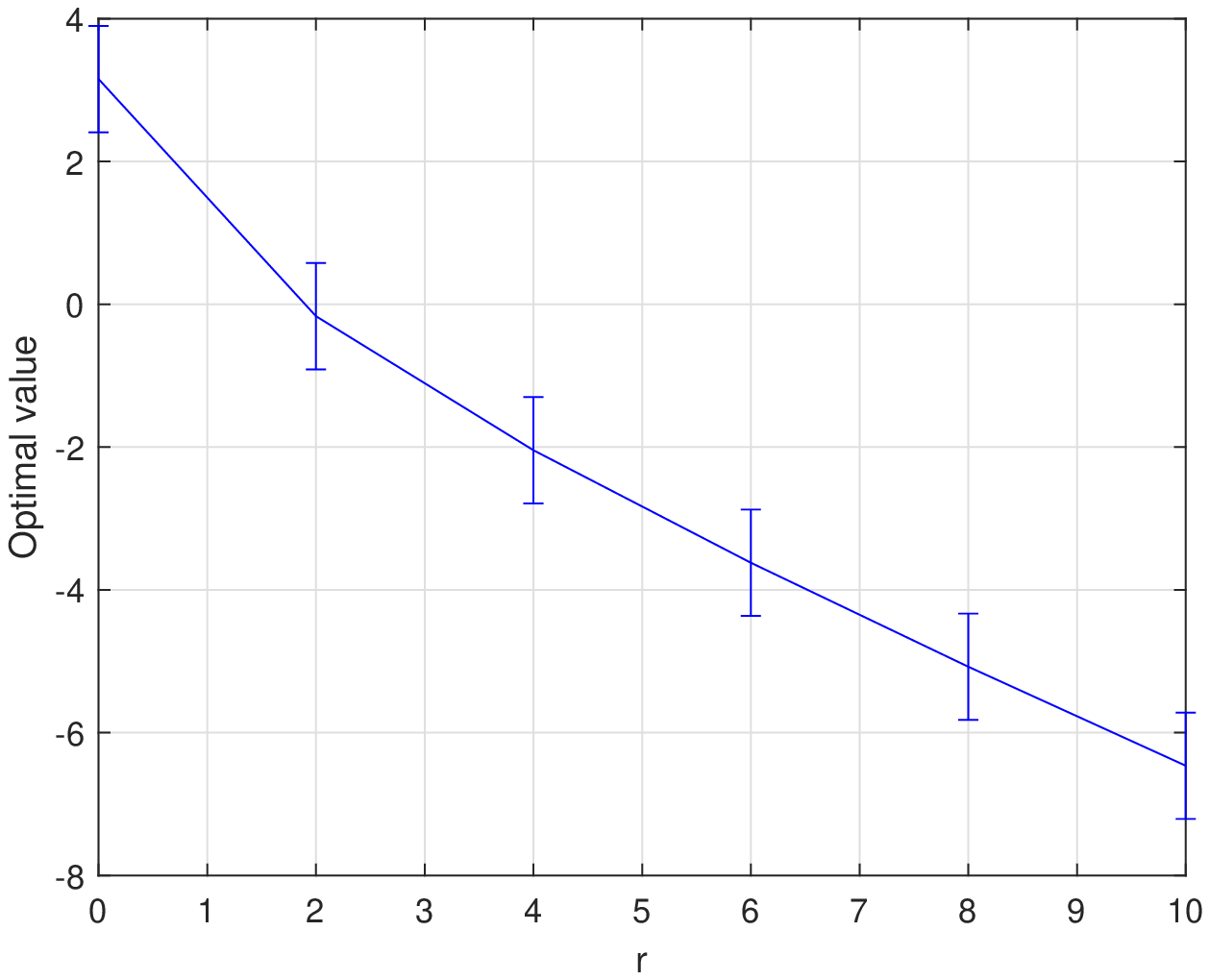}
	\end{minipage}
	\begin{minipage}[t]{0.32\textwidth}
		\centering
		\caption*{\scriptsize{Pareto (1,1.5)}}
		\includegraphics[width=5.5cm]{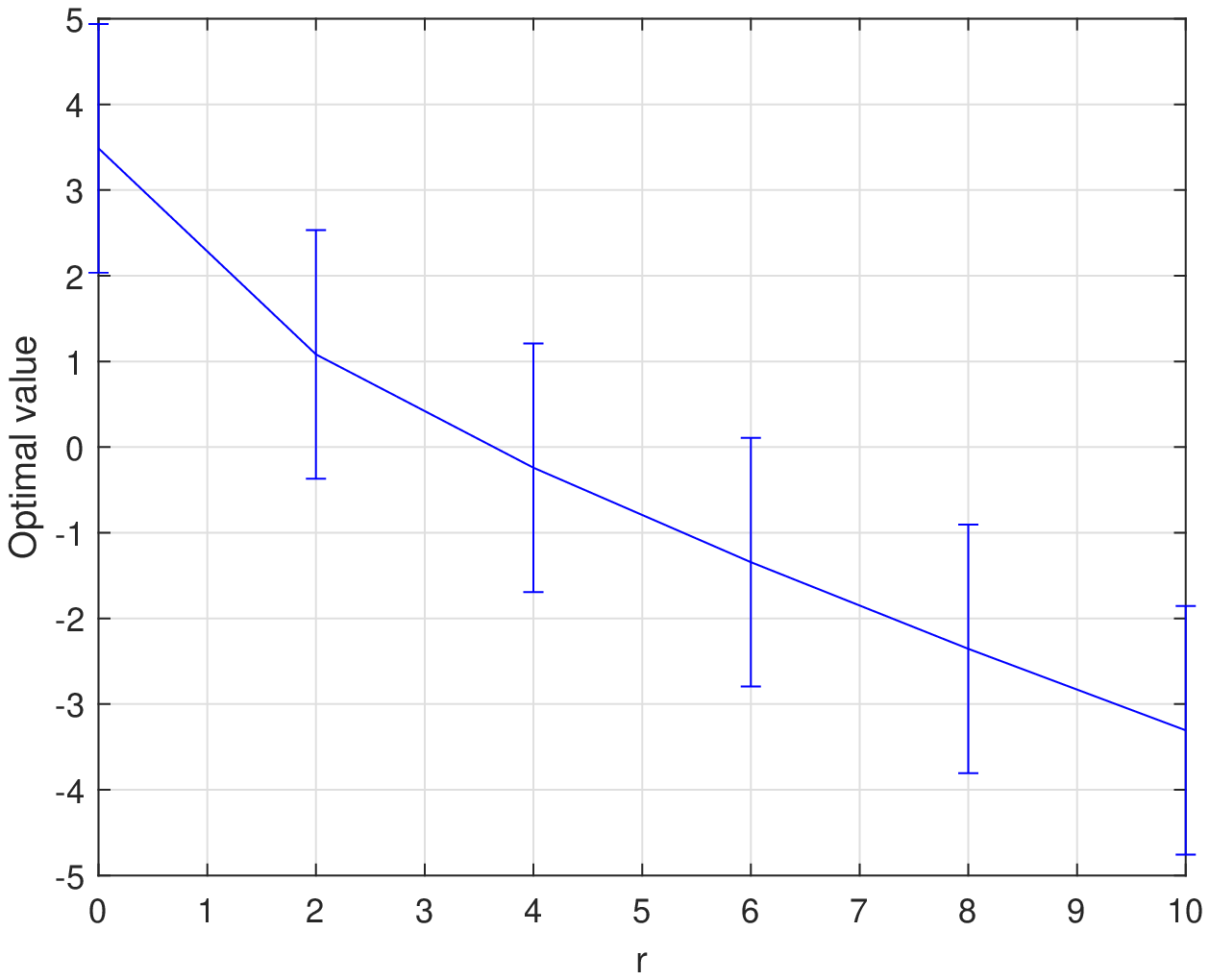}
	\end{minipage}
	\begin{minipage}[t]{0.32\textwidth}
		\centering
		\caption*{\scriptsize{Gamma (0.53,3)}}
		\includegraphics[width=5.5cm]{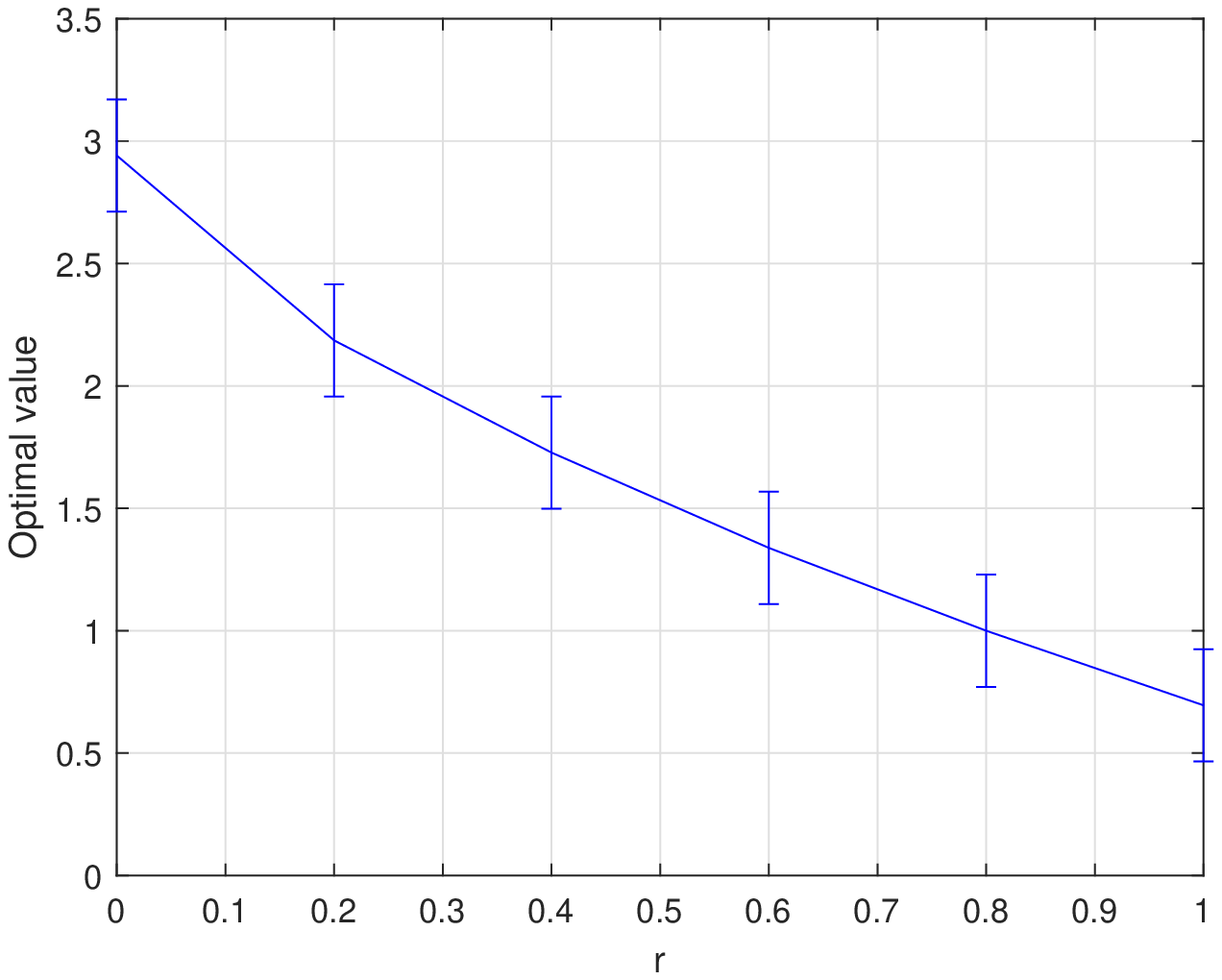}
	\end{minipage}
	\captionsetup{font={scriptsize}}
	\caption{Optimal values with heavy-tailed distributions}
	\label{fig-optimalvalue-distributions}
\end{figure}

\section{Conclusion}
\label{sec:conclu}

{\color{black}
In this paper we explore 
variations of 
the concept of optimized certainty equivalent
with a number of new inputs. 
First, we propose a  modified optimized certainty equivalent (MOCE) model
by considering the utility of present consumption.  
The optimal strategy (which balances the present and future consumption)
    is uniquely determined by the decision maker's risk preference rather than 
    by his/her utility representations (which is not unique). 
    The resulting MOCE value is positive homogeneous in $u$.  
    The MOCE is also in alignment with the  consumption models in economics.
Second, there is a distinction 
    between OCE and MOCE  in terms of the utility functions to be used in the model.
     In the classical OCE model, 
     it requires the utility function to 
     satisfy $u(0)=0$ and $1\in \partial u(0)$. 
     The new MOCE model does not require these conditions.
Third, we propose a preference robust version of the new MOCE model 
for the case that
the decision maker's true utility function is ambiguous. Ambiguity does exist in practice and 
this paper provides a comprehensive treatment
of the preference robust MOCE model from 
modelling to 
computational scheme and underlying theory.
Fourth, in the case that the proposed RMOCE model is applied to data-driven problems where
the underlying exogenous data (samples of $\xi$) are potentially contaminated, we derive sufficient conditions under which the RMOCE calculated with the data is statistically robust. 
Fifth, we outline potential extensions of 
the MOCE model from single decision making to multi-attribute decision making and point out potential applications in asset re-organization.  
In summary, 
this paper provides a new outlook
of OCE in both modelling and analysis,
which complement 
 the existing research 
 in the literature.
}
}



\bibstyle{plain}
\bibliography{RMOCE}

\end{document}